\documentclass[11pt]{amsart}

\usepackage[T1]{fontenc}
\usepackage{mathrsfs}
\usepackage{amssymb}
\usepackage{amsmath}
\usepackage{amsthm}
\usepackage{amsfonts}
\usepackage{fullpage}
\usepackage{enumerate}
\usepackage{tikz-cd} 
\usepackage{comment}
\usepackage{hyperref}
\usepackage{float}
\usepackage{verbatim}
\usepackage{enumitem}
\usepackage{fancyhdr}
\usepackage{todonotes}
\usepackage{svg}
\usepackage{multirow, url}
\usepackage{textcomp}
\usepackage{tikz-cd}
\usepackage{bbm}
\usepackage{mathtools}

\hypersetup{colorlinks}
\hypersetup{colorlinks=true, linkcolor=red, citecolor=blue, filecolor=blue, urlcolor=blue}

\newtheorem{thm}{Theorem}[section]

\newtheorem{lem}[thm]{Lemma}
\newtheorem{prop}[thm]{Proposition}
\newtheorem*{remark}{Remark}
\theoremstyle{definition}
\newtheorem{defn}[thm]{Definition}

\newenvironment{appendixproof}[1]
{\par\noindent\underline{\textit{Proof of {#1}.}}\\ }
{\hfill$\square$\par
\vspace{3mm}}

\numberwithin{equation}{section} 
\numberwithin{table}{section}

\newcommand{\nc}{\newcommand}
\newcommand{\rc}{\renewcommand}

\nc{\CC}{\mathbb{C}}
\nc{\EE}{\mathbb{E}}
\nc{\FF}{\mathbb{F}}
\nc{\Prob}{\mathbb{P}}
\nc{\QQ}{\mathbb{Q}}
\nc{\RR}{\mathbb{R}}
\nc{\ZZ}{\mathbb{Z}}
\rc{\SS}{\mathbb{S}}
\nc{\TT}{\mathbb{T}}

\nc{\argmax}{\on{argmax}}
\nc{\argmin}{\on{argmin}}
\nc{\ptx}{\mathcal{X}}
\nc{\pty}{\mathcal{Y}}
\nc{\ptz}{\mathcal{Z}}
\nc{\filt}{\mathcal{F}}
\nc{\Cechword}{\v Cech }
\nc{\Cech}{\mathcal{C}}
\nc{\rips}{\mathcal{R}}
\nc{\pd}{\mathsf{PD}}
\nc{\homology}{{\mathsf{H}}}
\nc{\ph}{{\mathsf{PH}}}
\nc{\prob}{{\mathsf{Prob}}}
\nc{\pointprocess}{{\mathcal{P}}}
\nc{\poissondist}{\mathsf{Poisson}}
\nc{\PL}{{\on{PL}}}
\nc{\diam}{\mathcal{L}}
\nc{\thick}{\mathcal{T}}
\nc{\sthick}{\mathcal{S}}
\nc{\diff}{\mathsf{d}}
\nc{\Diff}{\mathsf{D}}
\nc{\dist}{\mathsf{d}}
\nc{\sphere}{\mathbb{S}}
\nc{\frob}{{\mathsf{F}}}
\nc{\subdiv}{{\mathsf{div}}}
\nc{\sing}{{\mathsf{s}}}
\nc{\ord}{{\mathsf{ord}}}
\nc{\interior}{{\mathsf{int}}}
\nc{\diag}{{\mathsf{Diag}}}
\nc{\hess}{{\on{Hess}}}
\rc{\l}{\ell}
\nc{\supp}{\mathsf{supp}}

\nc{\GL}{{\mathsf{GL}}}
\rc{\epsilon}{\varepsilon}
\nc{\ball}{{\mathsf{B}}}
\nc{\func}{{\mathcal{H}}}
\rc{\hom}{{\mathsf{H}}}
\nc{\pivals}{{\Pi}}
\nc{\SO}{{\mathsf{SO}}}
\nc{\mytri}{{\mathsf{FK}}}
\nc{\simp}{\triangle}
\nc{\regsimp}[1]{{\triangle_{#1}}}
\nc{\ordsimp}[1]{{\triangle^{\mathsf{ord}}_{#1}}}
\nc{\cube}{\square}
\nc{\tanspc}{{\mathsf{T}}}
\nc{\normspc}{{\mathsf{N}}}
\nc{\permgrp}{\mathfrak{S}}
\nc{\permelt}{\mathfrak{s}}
\nc{\hausmeas}{{\mathscr{H}}}
\nc{\eqdef}{:=}
\nc{\indexset}[1]{{[ #1 ]_{\ZZ}}}
\nc{\equalinds}[1]{{\indexset{#1}}_{\mathsf{eq}}}
\nc{\constinds}[1]{{\indexset{#1}}_{\mathsf{con}}}
\nc{\eqtext}{{\mathsf{eq}}}
\nc{\simpclosd}{{\overline{\simp}_d^\ord}}
\nc{\FK}{{\mathsf{FK}}}
\nc{\fkteuclid}[1]{{\mathscr{R}^{#1}_{\FK}}}
\nc{\fktksimplex}[2]{{\mathscr{R}^{#1}_{\FK}(#2)}}
\rc{\l}{{\ell}}
\rc{\angle}{{\measuredangle}}
\nc{\reach}{{\varrho}}
\nc{\curv}{{\kappa}}
\nc{\euclidean}{{\mathsf{E}}}
\nc{\incone}{{\mathfrak{I}}}
\nc{\nincone}{{\mathfrak{N}}}
\nc{\rmax}{{\rho}}
\nc{\wtilde}{\widetilde}
\nc{\jacobian}{\mathsf{J}}
\nc{\univstar}{*}

\rc{\t}{\text}
\nc{\mbf}{\mathbf}
\nc{\mb}{\mathbb}
\nc{\tc}{\textcolor}
\nc{\mf}{\mathfrak}
\nc{\mc}{\mathcal}
\nc\on{\operatorname}
\nc{\vecbf}[1]{{\mathbf{#1}}}
\nc{\matbf}[1]{{\mathbf{#1}}}
\rc{\paragraph}[1]{\textbf{#1}}

\nc{\explainbox}[1]{
\noindent
\fbox{
    \parbox{\textwidth}
        {\vspace{1mm} {#1} \vspace{1mm}}
    }
\vspace{3mm}}

\title{Universal topological statistics on \\
triangulated singular spaces}

\begin{document}

\author[U. Lim]{Uzu Lim}
\address{Queen Mary University of London}
\email{sung.h.lim@qmul.ac.uk}

\author[O. Bobrowski]{Omer Bobrowski}
\email{o.bobrowski@qmul.ac.uk}

\author[P. Skraba]{Primoz Skraba}
\email{p.skraba@qmul.ac.uk}

\begin{abstract}
    We prove a universality theorem for random persistent homology over a class of triangulable spaces. More precisely, let $M \subset \RR^D$ be a compact $C^2$-triangulable space satisfying a geometric quality condition and let $f: M \to \RR$ be a probability density. Then the expected persistence ratio measure computed from the \v Cech or Vietoris-Rips complex of a Poisson point process with intensity $nf$ has a universal limit independent of $(M, f)$. Since smooth manifolds, algebraic varieties, semialgebraic sets and Whitney stratified spaces are all triangulable spaces, our theorem applies to a large class of non-Euclidean spaces. Beyond persistent homology, our proof covers a general class of scale-invariant functionals. It relies on a geometric transfer method that adapts constructions in Euclidean space to triangulable spaces through successive approximations by Freudenthal--Kuhn triangulations, and control of interference across singular strata.
\end{abstract}

\maketitle

\section{Introduction}
Universality describes the phenomenon where applying a transformation collapses a whole family of mathematical objects into a single universal object. The central limit theorem is a classical example where various probability distributions collapse into the normal distribution upon taking the sample mean and rescaling. In random matrix theory, universality is exhibited in Wigner's semicircle law and its subsequent refinements \cite{rmt_univ1, rmt_univ2, rmt_univ3}. Other examples include extreme value theory, the KPZ universality class \cite{kpz1}, and geometric probability \cite{penrose_yukich1, penrose_yukich2}. 

Recently, persistent homology emerged as another surprising example of an object exhibiting universality. Persistent homology assigns homology groups to a discrete metric space across multiple scales of connectivity. It is a key tool of topological data analysis (TDA); for a survey, see \cite{tda_overview1, tda_overview2, tda_overview3}. Applications include neuroscience \cite{gridcell_1, gridcell_5}, quantum physics \cite{ph_quantum1, ph_quantum2}, cell biology \cite{ph_cell1, ph_cell2} and symplectic topology \cite{ph_puremath2, ph_puremath3}. To observe universality, one computes the death/birth ratios arising from the persistent homology of a random point process, and notices that the empirical distribution of the ratios is independent of the law of the point process. This was initially discovered experimentally \cite{univ_experiment}, and later proven for probability densities supported on the Euclidean space \cite{univ_euclidean}.

In this paper we take the universality theorem of persistent homology and generalize its domain from Euclidean spaces to \textit{triangulable spaces}. A broad class of spaces, including those with non-zero, non-constant curvature and singularities are triangulable, including smooth manifolds, algebraic varieties, semialgebraic sets, and Whitney stratified spaces. It is important to note that this work only deals with transient cycles which arise due to randomness in the underlying point process rather than those which are a result of non-trivial topology of the support -- see \cite{univ_experiment} for a formal definition. 

Our key contributions are two-fold:
\begin{enumerate}
    \item We prove the universality theorem of persistence ratios and scale-invariant functionals over \textit{admissible triangulable spaces} (Theorem \ref{thm:A} and \ref{thm:A_general})
    \item We establish methods to transfer geometric statistics from a Euclidean space to triangulable spaces using approximations by Freudenthal-Kuhn triangulation and a control of singular strata (Propositions \ref{prop:B} and \ref{prop:C}).
\end{enumerate}

The admissible triangulable spaces for which our results hold are characterized by the  condition that neighbouring simplices must meet at a positive angle (see \eqref{eqn:separation_angle}). In particular, we do not allow cusp singularities in the triangulable space. While the main objective of the general tools developed in the paper was to prove the universality theorem of persistent homology, we expect that the universality theorem for general functionals is applicable for a variety of functionals. We also expect that the geometric transfer framework we develop here can be applied to other constructions in stochastic geometry and topology.

\medskip
\noindent\paragraph{Structure of the paper.} We provide the background material in Section \ref{section: preliminary}. We then state our main results in Section \ref{section: main_statements}. Sections \ref{section: proofs_main}--\ref{section: proofs_technical} contain the proofs. We first prove the main results (Section \ref{section: proofs_main}), then the geometric transfer methods (Section \ref{section: proofs_geometric}), and the more technical lemmas (Section \ref{section: proofs_technical}). The Appendix contains a derivation of the FK-triangulation theory and important nuances regarding the necessity of simplex quality and vertex ordering.

\medskip
\noindent\paragraph{Related work.} Stochastic topology studies the probabilistic behavior of topological invariants; for a survey, see \cite{stochastic_topology_survey1,stochastic_topology_survey2}. In particular, Hiraoka, Shirai, and Trinh established the convergence of normalized random persistence diagrams to deterministic limiting measures \cite{hiraoka_shirai_trinh}. Subsequent work refined probabilistic and analytic structure of the limiting objects \cite{divol_polonik, skraba_yogesh, owada_bobrowski, shirai_suzaki, krebs_polonik}. The universality phenomena of persistence ratios have been reported experimentally \cite{univ_experiment}, proven theoretically for Euclidean spaces \cite{univ_euclidean}, and also explored for the case of nearest-distance based dimension estimators \cite{univ_l2nn}. In stochastic geometry, universality behavior was first observed in the theory of stabilizing functionals developed by Penrose and Yukich \cite{penrose_yukich1, penrose_yukich2}. 

Triangulation is a central object in pure mathematics and computational geometry. One prominent line of research is studying classes of manifolds that admit triangulations. Smooth manifolds of any dimension are triangulable \cite{cairns, whitehead}. Topological $n$-manifolds are triangulable when $n \le 3$ while counterexamples exist for $n \ge 4$ \cite{casson_akbulut_mccarthy, kirby_siebenmann, manolescu}. Many classes of non-manifold spaces are also triangulable, such as semialgebraic sets \cite{hironaka_triangulation} and Whitney stratified spaces \cite{goresky_triangulation}. In computational geometry, triangulations of dimensions 2 and 3 are studied often with applications in mind. Key constructions include the Delaunay triangulation \cite{delaunay1, sibson} and red-green refinement \cite{bank}. One line of work relevant to us are the Coxeter triangulations and Freudenthal-Kuhn triangulation \cite{cfk1, cfk2, cfk3, coxeter_quality}. Our construction of adapting the Freudenthal-Kuhn triangulation to arbitrary simplicial complexes is also known as the \textit{edgewise subdivision} \cite{edelsbrunner_grayson}. While pure topology and computational geometry often have little overlap, this paper draws influence from both fields.

Our focus on spaces with singularities is motivated by the complexity of real-world data. The manifold hypothesis states that a generic real-world dataset is distributed near a low-dimensional smooth submanifold, but evidence of singularities have been found in image datasets \cite{union_of_manifolds_brown}, token embeddings for large language models \cite{token_mixed_dim}, token embeddings for reinforcement learning \cite{curry_speranzon}, and single-cell transcriptomic data \cite{packer_celegans}. Such datasets need not be well described by a single smooth manifold. Intersections, branching, and changes in local dimension can all produce singular behavior. Algorithms have been proposed to study the singularities of such datasets, using persistent homology and kernel techniques \cite{singdet1, singdet2, singdet3}. This motivates our extension of universality beyond manifolds to the much broader class of triangulable spaces.

\medskip
\noindent\textbf{Acknowledgement.} We would like to thank Mathijs Wintraecken for helpful discussions and suggestions regarding triangulations and simplex quality.

\section{Preliminary notions}
\label{section: preliminary}

We will first cover some preliminary notions and conventions (see also Appendix \ref{appendix: notations}):
\begin{itemize}[noitemsep]
    \item $\dist_\euclidean$ is the Euclidean distance, and $\dist_M$ is the intrinsic distance on a metric space $M$.
    \item $\dist_M (x, S) = \inf_{y \in S} \dist_M (x, y)$, where $M$ is a metric space, $x \in M$, and $S \subseteq M$.
    \item $\ball(x, r)$ is the closed ball of radius $r$ centred at $x$, and $\ball(S, r) = \cup_{x \in S} \ball(x, r)$.
    \item $\|v\|_p$ is the $p$-norm of a vector and also $\|v\|=\|v\|_2$. 
    \item $\sing_{\min}(A)$ is the minimum singular value of a real matrix $A$.
\end{itemize}

\medskip
\noindent\paragraph{Triangulation.} A \textit{simplicial complex} is a topological space obtained by gluing affine simplices compatibly across their faces. An affine simplex is an affine linear embedding of the standard simplex $\regsimp{d} = \{ (t_0, \ldots t_d) \>|\> t_0 + \cdots + t_d = 1, \> t_i \ge 0 \}$. A simplicial complex is said to be \textit{purely $d$-dimensional} if all of its maximal simplices are $d$-dimensional. A \textit{triangulation} $\phi: K \to M$ is a homeomorphism from a simplicial complex $K$ to a metric space $M$. If a triangulation $\phi$ is a $C^k$ map on each closed simplex of $K$, then we say that it is a $C^k$ triangulation. Here, $\phi$ is said to be a $C^k$ map on a closed simplex $\sigma$ if $\phi$ has $k$-th continuous partial derivatives on some open neighbourhood of $\sigma$.

\medskip
\noindent\paragraph{Simplex quality.} The quality of a simplex is a measure of how far the simplex is from being degenerate. There are many notions of simplex quality; see \cite{coxeter_quality} for a comparison for the Coxeter triangulations of the Euclidean space. Given an affine simplex $\sigma \subset \RR^D$ with vertices $(u_0, \ldots u_d)$, we define its \textit{diameter} $\diam(\sigma)$, \textit{thickness} $\thick(\sigma)$, and \textit{S-width} $\sthick(\sigma)$ as follows:
\begin{align} \label{eqn:simplex_quality}
    \diam(\sigma) \eqdef \max_{i, j} \dist_\euclidean(u_i, u_j), \qquad \thick(\sigma) \eqdef \frac{\min_{i} \dist_\euclidean (u_i, P_i)}{\diam(\sigma)}, \qquad \sthick(\sigma) \eqdef \max_{0 \le i \le d} \sing_{\min}(A_{\sigma, i}),
\end{align}
where $P_i$ is the affine space spanned by the facet of $\sigma$ opposite to the vertex $u_i$ and $A_{\sigma, i} = [u_0 - u_i, \ldots u_d - u_i] \in \RR^{D \times d}$ is the matrix of vertex displacements. It is well known that the diameter is also equal to the longest pairwise distance of points on $\sigma$. Let us also note that for each $\lambda > 0$, we have $\diam(\lambda \cdot \sigma) = \lambda \cdot \diam(\sigma)$, $\thick(\lambda \cdot \sigma) = \thick(\sigma)$, and $\sthick(\lambda \cdot \sigma) = \lambda \cdot \sthick(\sigma)$.

\medskip
\noindent\paragraph{Spatial Poisson process.} Let $f: \RR^D \rightarrow \RR_{\ge 0}$ be a measurable function such that $\nu = \int_{\RR^D} f(x) \diff x$ satisfies $0 < \nu < \infty$. The \textit{Poisson point process} with intensity $f$, denoted $\pointprocess_{f}$, can be defined as the i.i.d.~sample of size $N$ drawn from the probability density $\nu^{-1} f$, where $N$ is a Poisson random variable with parameter $\nu$.

\medskip
\noindent\paragraph{Reach.} Given a set $S \subseteq \RR^D$ and a point $x \in \RR^D$, let $\pi_S(x) \subset S$ be the set of all $s \in S$ that satisfies $\dist_\euclidean(x, S) = \dist_\euclidean(x, s)$. Then the \textit{reach} $\reach_S$ is defined as the supremum of all $r>0$ so that the following holds: For every point $x$ satisfying $\dist_\euclidean(x, S) \le r$, we have $|\pi_S(x)| = 1$. For our purposes, reach is a useful tool that controls the second-order deviation of a geodesic on a manifold:

\begin{lem}[Proposition 6.3 of \cite{niyogi_smale_weinberger}] \label{lem:geodesic_control}
    Let $M \subset \RR^D$ be a compact $C^2$ manifold. Suppose that $x, y\in M$ are connected by a geodesic $\gamma_{xy}$ of length $s$, and write $r = \|x-y\| $. Define $v$ as the initial velocity of $\gamma_{xy}$ at $x$. If $r \le (\sqrt 2 - 1) \reach_M$, the following are true:
    \begin{align*}
        & s - \frac{s^2}{2 \reach_M} \le r \le s \le r + \frac{r^2}{\reach_M}, \qquad\text{and}\qquad \|y - (x + s v) \| \le \frac{r^2}{\reach_M}.
    \end{align*}
\end{lem}

\medskip
\noindent\paragraph{Persistent homology.} Persistent homology is a method to assign homology groups to a nested family of topological spaces. In this paper, such a nested family is given by continuous approximations of a discrete set over multiple scales. Let $K_\bullet = (K_1, K_2, \ldots, K_m)$ be a filtered simplicial complex, meaning that $K_i \subseteq K_{i+1}$ for all $i$. Then the \textit{$k$-th persistence module} of $K_\bullet$ with the field of coefficients $\FF$ is the $\FF[t]$-module defined as $\ph_k(K_\bullet) = \homology_k(K_1; \FF) \oplus \cdots \oplus \homology_k(K_m; \FF)$. The $\FF[t]$-module structure is given by $t \cdot \omega = \iota^{i}_*(\omega)$ for each $\omega \in \homology_k(K_i; \FF)$, where $\iota^{i}: K_i \hookrightarrow K_{i+1}$ is the inclusion map and $\iota^{i}_*$ is the induced map on homology groups. By applying the structure theorem of modules over a principal ideal domain, one obtains the following decomposition of a persistence module into interval modules $\mathbb{I}_\FF[b, d)$ (see \cite{crawley_decomposition} for a more detailed explanation):
\begin{align*}
    \ph_k(K_\bullet) \cong \bigoplus_{i \in \mathcal{I}} \mathbb{I}_\FF[b_i, d_i).
\end{align*}
Intuitively, a persistence module records how homological features evolve as the scale parameter increases. Each interval $[b,d)$ in its decomposition represents a feature that first appears, or is \emph{born}, at scale $b$ and disappears, or \emph{dies}, at scale $d$. Consequently \textit{persistence diagram} and \textit{persistence ratios} are defined by collecting all data of $(b, d)$ and $d/b$ appearing in the above decomposition:
\begin{align*}
    \pd_k(K_\bullet) \eqdef & \bigg\{ (b_i, d_i) \>\bigg|\> i \in \mathcal{I}, \> \ph_k(K_\bullet) \cong \bigoplus_{i \in \mathcal{I}} \mathbb{I}_\FF[b_i, d_i) \bigg\}, \\
    \Pi_k(K_\bullet) \eqdef & \bigg\{ \>\>\>\> \frac db \>\>\>\>\>\> \bigg|\> (b,d) \in \pd_k(K_\bullet) \bigg\}.
\end{align*}

The \textit{Vietoris-Rips and \v Cech filtrations} are standard methods used to produce continuous approximations of a discrete  metric space $\ptx$. Let $M$ be a metric space and let $\ptx \subset M$ be a finite subset. The \textit{Vietoris-Rips complex} $\rips_r(\ptx)$ is constructed by inserting $k$-simplex to each subset $\{x_0, \ldots x_k \} \subset \ptx$ whose pairwise distances are all within $r$, i.e. $\dist_M(x_i, x_j) \le r$, $\forall (i,j)$. The \textit{\v Cech complex} $\Cech_r(\ptx)$ is constructed by inserting $k$-simplex to each subset $\{x_0, \ldots x_k \} \subset \ptx$ such that balls of radius $r$ centred at $x_i$ have a nonempty intersection, i.e. $\cap_i \ball(x_i, r) \neq \emptyset$. The Rips and \v Cech \textit{filtration} are defined as $\rips_\bullet(\ptx) = \{ \rips_r(\ptx) \}_{r \ge 0}$ and $\Cech_\bullet(\ptx) = \{ \Cech_r(\ptx) \}_{r \ge 0}$.

\section{Statement of main results}
\label{section: main_statements}

\subsection{The case of Euclidean space}

We first recall results of \cite{univ_euclidean}, where universality theorems were proven over the setting of the Euclidean space. Given a probability density $f: \RR^d \rightarrow \RR_{\ge 0}$ and either the Vietoris-Rips or the \v Cech filtration $\filt$, the \textit{expected persistence measure} $\overline{\Pi}_{k, \filt}(nf)$ is a measure on $\RR$ is defined as follows. For every Borel set $B \subseteq \RR$, we set:
\begin{align} \label{eqn:expected_persistence_measure}
    \overline{\Pi}_{k, \filt}(nf)(B) \>\eqdef\>\> & \EE\> \big| \Pi_k (\filt_\bullet' (\pointprocess_{nf} )) \cap B \big|,
\end{align}
where $\filt_\bullet'$ is the curtailed filtration that grows the filtration radius up to $\rmax_n$. Here $\rmax_n$ is a radius cutoff parameter that satisfies $\rmax_n \to 0$ and $n \rmax_n^d \to \infty$ as $n \to \infty$.

The universality results hold for a \textit{good density}, defined as a probability density $f: \RR^d \to \RR_{\ge 0}$ satisfying the following conditions. Firstly, the support $\supp(f)$ is compact. Secondly, either $\inf_{x \in \supp(f)} f(x) > 0$ holds, or there exist constants $c_\pm, q > 0$ such that $c_- (\delta(x))^q \le f(x) \le c_+ (\delta(x))^q$. Here, $\delta(x) = \dist_\euclidean \big( x, f^{-1}(0) \cap \supp(f) \big)$ is the distance of $x$ to the boundary of support. While the universality theorem was proven in \cite{univ_euclidean} for certain densities of unbounded support, we only focus on compact support in this paper. 

\begin{thm}[Universality of persistence on Euclidean space \cite{univ_euclidean}] \label{thm:univ_euclidean}
    Let $f: \RR^d \rightarrow \RR_{\ge 0}$ be a good density and let $\filt$ be either the Rips or \v Cech filtration. Then there exists a measure on $[1,\infty)$, denoted $\overline{\Pi}_{k,\filt, d}^\univstar$, such that the following holds independently of $f$:
    \begin{gather*}
        \lim_{n \rightarrow \infty} \frac 1n \overline{\Pi}_{k, \filt} ( nf ) = \overline{\Pi}_{k,\filt,d}^\univstar.
    \end{gather*}
\end{thm}

Theorem \ref{thm:univ_euclidean} follows from a more general result that holds for many scale-invariant functionals. In this paper, a \textit{functional} $\func$ refers to a function that takes a finite set $\ptx \subset \RR^D$ and outputs a real number $\func(\ptx) \in \RR$. Given a non-negative measurable function $f: \RR^d \to \RR_{\ge 0}$ with $\int_{\RR^d} f(x) \diff x \in (0, \infty)$, we will denote the expected value of the functional evaluted with the Poisson point process $\pointprocess_f$ as $\overline{\func}(f) = \mathbb E \big[ \func( \pointprocess_{f} ) \big]$. Note that the argument of $\func$ is a set, while the argument of $\overline{\func}$ is a function:
\begin{align*}
    \func(\ptx), \quad \text{vs.} \quad \overline{\func}(f)
\end{align*}

\paragraph{Conditions for functionals.} Let $\func_n$ be a functional that depends on an optional control parameter $n$. In the case of the persistence functional, the parameter $n$ in $\func_n$ is used for the radius cutoff $\rmax_n$. Then the following conditions are of interest to us:
\begin{enumerate}[label=(\arabic*)]
    \item \underline{Rigidity.} Let $\ptx \subset \RR^D$ be a finite subset. For any vector $x \in \RR^D$ and rotation $R \in \mathsf{SO}(D)$,
    \begin{align*}
        \func_n (x + R(\ptx)) = \func_n (\ptx).
    \end{align*}
    \item \underline{Control homogeneity.} Let $f$ be a probability density. Then for any $c>1$,
    \begin{align*}
        \lim_{n \to \infty} \frac{\overline{\func}_{cn}(nf)}{\overline{\func}_n(nf)} = 1.
    \end{align*}
    \item \underline{Scale-invariant reference.} Let $f_\epsilon = \epsilon^{-d} \mathbf{1}_{[0,\epsilon]^d}$. There exists a number $\overline{\func}^*$ so that the following holds for every $\epsilon \in (0, 1]$:
    \begin{align*}
        \lim_{n \to \infty} \frac1n \overline{\func}_n(n f_\epsilon) = \overline{\func}^*.
    \end{align*}
    \item \underline{Continuity.} Let $f$ be a probability density and let $\{f_i\}$ be a sequence of non-negative functions such that $f_i \le f_j$ whenever $i \le j$. Suppose there is a pointwise convergence $f_i \to f$. Then the following holds:
    \begin{align*}
        \lim_{i \to \infty} \> \limsup_{n \to \infty} \> \frac1n \bigg| \overline{\func}_n(nf) - \overline{\func}_n(n f_i) \bigg| = 0.
    \end{align*}
    \item \underline{Cube additivity.} Let $Q = [0, 1]^d$ and let $Q_1, \ldots Q_{m^d}$ be the cubes of sidelength $1/m$ obtained by partitioning $Q$ into smaller cubes. Let $f = \mathbf{1}_{Q}$ and $f_i = c_i m^d \cdot \mathbf{1}_{Q_i}$, where $c_i \ge 0$ and $\sum_i c_i = 1$. Then,
    \begin{align*}
        \lim_{n \to \infty} \frac1n \bigg| \overline{\func}_n(n f) - \sum_{i=1}^{m^d} \overline{\func}_n(nf_i) \bigg| = 0.
    \end{align*}
\end{enumerate}

We note that the condition (1) as stated includes rotations, which was not required in \cite{univ_euclidean}. However, persistent homology is invariant to rotations of the underlying set. Assuming the above conditions, the following general result was proven in \cite{univ_euclidean}.

\begin{thm}[Universality of general functionals on Euclidean space \cite{univ_euclidean}] \label{thm:univ_euclidean_general}
    Let $f: \RR^d \rightarrow \RR_{\ge 0}$ be a good density and let $\func_n$ be a functional that satisfies the conditions (1)--(5). Then there exists a real number $\overline{\func}^*$ such that the following holds independently of $f$,
    \begin{gather*}
        \lim_{n \rightarrow \infty} \frac 1n \overline{\func}_n ( nf ) = \overline{\func}^\univstar.
    \end{gather*}
\end{thm}

\begin{remark}
    To derive Theorem \ref{thm:univ_euclidean}, one applies Theorem \ref{thm:univ_euclidean_general} for each choice of $\alpha \in (1, \infty)$:
    \begin{align*}
        \func_n(\ptx) = \big| \Pi_k(\filt'_\bullet(\ptx)) \cap [\alpha, \infty) \big|,
    \end{align*}
    where $\filt_\bullet'$ is the curtailed filtration that grows the filtration radius up to $\rmax_n$.
\end{remark}

\medskip
\subsection{The case of triangulable spaces}

We now explain how the universality theorems from \cite{univ_euclidean} (Theorems \ref{thm:univ_euclidean}, \ref{thm:univ_euclidean_general}) extend to a class of spaces that we call \textit{admissible triangulable spaces}. This means that we assume compactness and exclude cusp singularities, and we define this using an angle condition on normal cones. 

\medskip
\noindent \paragraph{Admissible triangulations.} Let $\sigma$ be an affine simplex and let $\phi: \sigma \hookrightarrow \RR^D$ be a $C^2$ embedding, whose image we denote by $\tilde{\sigma} = \phi(\sigma)$. Given a point $x \in \partial \tilde{\sigma}$, let $\tilde{\sigma}_x$ denote the smallest closed facet of $\tilde{\sigma}$ that contains $x$. The \textit{normal inward cone} $\nincone_x \tilde \sigma$ is defined as the set of all $v \in \RR^D$ that is perpendicular to $\tilde \sigma_x$ and also $x + \epsilon v \in \tilde{\sigma}$ some $\epsilon > 0$. The \textit{strata separation angle} $\Theta_\phi$ is defined as the following infimum, taken over $d$-simplices $\sigma_1 \neq \sigma_2 \in K$ and interior points $x \in (\sigma_1 \cap \sigma_2)^\circ$:
\begin{align} \label{eqn:separation_angle}
    \Theta_\phi \eqdef 
    \inf_{\sigma_1, \sigma_2, x} \dist_{\SS} \big( \nincone_{x} \tilde \sigma_1, \nincone_{x} \tilde \sigma_2 \big),
\end{align}
where $\dist_{\SS}(A_1, A_2)$ is the spherical distance between the sets $A_1 \cap \SS^{D-1}$ and $A_2 \cap \SS^{D-1}$. Finally, we say that a $C^2$-triangulation $\phi: K \to M \subset \RR^D$ is \textit{admissible} if $M$ is compact, $K$ is a purely $d$-dimensional simplicial complex, and satisfies $\Theta_\phi > 0$.

\begin{thm}[Universality of persistence on triangulable spaces] \label{thm:A}
    Let $M \subset \RR^D$ be an admissibly triangulable space, let $f: M \rightarrow \RR_{\ge 0}$ be a good density, and let $\filt$ be either the Vietoris-Rips or \v Cech filtration. Then the following holds independently of $M$ and $f$:
    \begin{gather*}
        \lim_{n \rightarrow \infty} \frac 1n \overline{\Pi}_{k, \filt} ( nf ) = \overline{\Pi}_{k,\filt,d}^\univstar,
    \end{gather*}
    where $\overline{\Pi}_{k,\filt,d}^\univstar$ is the same universal limit as in Theorem \ref{thm:univ_euclidean}. In particular, this result holds for any compact $C^2$-embedded manifold $M$.
\end{thm}

\begin{remark}
    This result requires the radius cutoff $\rmax_n$ to additionally satisfy $n\rmax_n^d = o(n^{1/(d^2+d+1)})$ (see Lemma \ref{prop:sigma_control}). This choice of $\rmax_n$ is mainly technical and carries no significance. In particular, in  \cite{univ_euclidean,hiraoka_shirai_trinh} it is shown that for the transient cycles we study here, the vast majority of them exist inside the thermodynamic limit ($n r^d = O(1)$), and 
    the number of cycles generated at larger radii is negligible. Therefore, as long as $n\rho_n^d\to \infty$ the results should remain the same. Additionally, it is possible to derive an analogue of Theorem \ref{thm:A} and a law of large numbers if we use a binomial point process instead of a Poisson point process. As this is a matter of additional technical steps, we do not include this result in this paper.
\end{remark}

Theorem \ref{thm:A} is a corollary of a result that holds for more general functionals. In addition to the conditions (1)--(5) stated in the case of Euclidean space, we need two more conditions for the case of triangulable spaces. One condition is a stricter version of the cube additivity (and thus replaces it), and another is on perturbations of the Euclidean metric. In the conditions below, we will use the following notation to specify the dependence of a functional $\func$ on a metric function $\dist$,
\begin{align} \label{eqn:functional_metric_dependence}
    \func(\ptx; \dist) \>\>\text{ and }\>\> \overline{\func}(nf; \dist) = \EE\big[ \func(\pointprocess_{nf}; \dist) \big].
\end{align}

\begin{enumerate}[label=(\arabic*),start=6]
    \item \underline{Simplex additivity.} The following holds for any admissible triangulation $\phi: K \rightarrow M$ and a probability density $f: M \rightarrow \RR_{\ge 0}$:
    \begin{align} \label{eqn:simplex_additivity1}
        \lim_{n \rightarrow \infty} \frac1n \bigg| \overline{\func}_n(nf) - 
        \sum_{\dim \sigma= d} \overline{\func}_n (nf|_{\phi(\sigma)}) \bigg| = 0.
    \end{align}
    \item \underline{Metric stability.} For every affine simplex $\sigma$ equipped with the Euclidean metric $\dist_\euclidean$ and a probability density $f: \sigma \to \RR_{\ge 0}$, there is a function $\mathcal{E}$ such that $\lim_{t \downarrow 0} \mathcal{E}(t) = 0$ and the following holds for any metric $\dist$ on $\sigma$:
    \begin{align} \label{eqn:metric_stability}
        \limsup_{n \rightarrow \infty} \frac1n \big| \overline{\func}_n(nf; \dist) - \overline{\func}_n (nf; \dist_\euclidean)  \big| \le \mathcal{E}(\Delta(\dist, \dist_\euclidean)),
    \end{align}
    where $\Delta(\dist_1, \dist_2)$ is the infimum of all $\delta>0$ for which the following holds for all $x, y \in \sigma$:
    \begin{align} \label{eqn:metric_distortion}
        (1+\delta)^{-1} \dist_1(x,y) \le \dist_2(x,y) \le (1+\delta) \dist_1(x,y).
    \end{align}
\end{enumerate}

We now state our second main result, which is a generalization of Theorem \ref{thm:A} that holds for well-behaved scale-invariant functionals.

\begin{thm}[Universality of general functionals over triangulable spaces] \label{thm:A_general}
    Let $M \subset \RR^D$ be an admissibly triangulable space and let $f: M \rightarrow \RR_{\ge 0}$ be a good density. Let $\func_n$ be a functional that satisfies the conditions (1)--(7). Then there exists a real number $\overline{\func}^*$ so that the  following holds independently of $M$ and $f$:
    \begin{align*} 
        \lim_{n \rightarrow \infty} \frac1n \overline{\func}_n (nf) = \overline{\func}^*.
    \end{align*}
\end{thm}

\subsection{Geometric transfer}

We develop tools for transferring geometric statistics from Euclidean spaces to a triangulable space. First we explain a heuristic version of the strategy. Suppose that $\ptx \subset M$ is a finite subset of a triangulable space $M \subset \RR^D$ and $\phi: K \to M$ is a triangulation such that $\dist_{K}(x, y) \approx \dist_M(\phi(x), \phi(y))$. Our strategy is to derive Theorem \ref{thm:A_general} from Theorem \ref{thm:univ_euclidean_general} by the following successive approximations, and controlling the error in each step,
\begin{align} \label{eqn:heuristic_approximation}
    \func(\ptx) \approx \sum_\sigma \func(\ptx \cap \tilde{\sigma}) \approx \sum_\sigma \func(\phi^{-1}(\ptx \cap \tilde{\sigma})),
\end{align}
where $\tilde{\sigma} = \phi(\sigma)$ and $\sigma$ ranges over the maximal simplices of $K$.

\medskip
\noindent\paragraph{Step 1. Divide and conquer.} In the approximation $\func(\ptx) \approx \sum_\sigma \func(\ptx \cap \tilde \sigma)$, a key part of the error estimates is bounding the number of pairs $(x_1, x_2)$ that lie on two embedded simplices, i.e. $x_1 \in \phi(\sigma_1)$ and $x_2 \in \phi(\sigma_2)$ for two simplices $\sigma_1, \sigma_2 \in K$. We define the \textit{$\phi$-skeleton} of a triangulation $\phi: K \to M$ as the following set obtained by removing interiors of all maximal simplices $\sigma \in K$,
\begin{align} \label{eqn:phi_skeleton}
    \partial_\phi M \eqdef M \backslash (\cup_\sigma \phi(\sigma^\circ)).
\end{align}
Then a pair of points lying in two distinct embedded simplices must lie close to the skeleton.

\begin{prop} \label{prop:C}
    Let $\phi: K \rightarrow M \subset \RR^D$ be an admissible triangulation and let $\sigma_1, \sigma_2 \in K$ be distinct $d$-simplices. Given points $x_1 \in \phi(\sigma_1)$, $x_2 \in \phi(\sigma_2)$, the following holds for $i=1,2$ if $\|x_1-x_2\|$ is sufficiently small:
    \begin{align*}
        \dist_\euclidean (x_i, \partial_\phi M) \le \frac{2 \|x_1 - x_2\|}{\sin(\Theta_\phi/4)}.
    \end{align*}
\end{prop}
This is proved in Subsection \ref{subsection: interference_control} using a careful control of the geodesic tubular neighbourhood. The next step is to show that the tube volume is small.

\begin{lem}\label{lem:skeleton_volume}
    Let $\phi: K \rightarrow M \subset \RR^D$ be a purely $d$-dimensional $C^2$ triangulation, and let $\hausmeas^d$ be the $d$-dimensional Hausdorff measure. Then for some $b_\phi^{\pm} > 0$ and small enough $r>0$,
    \begin{align*}
        b_\phi^- \cdot r \le \hausmeas^d \big( \ball ( \partial_\phi M, r) \cap M \big) \le b_\phi^+ \cdot r.
    \end{align*}
\end{lem}
\begin{proof}
    It suffices to show this per each $d$-dimensional simplex $\sigma \in K$. First, note that the inequality is obvious for a linear embedding $\phi: \sigma \rightarrow \RR^D$. Secondly, the area formula of geometric measure theory \cite{federer}, together with the existence of a lower and upper bound for the Jacobian of $\phi$ implies the claimed inequality.
\end{proof}

\medskip
\noindent \paragraph{Step 2. Approximation by subdivision.} The second step of the approximation \eqref{eqn:heuristic_approximation} is $\func(\ptx \cap \tilde \sigma) \approx \func(\phi^{-1}(\ptx \cap \tilde \sigma))$, and we need to control the error caused by flattening a curved simplex $\tilde \sigma = \phi(\sigma)$ to the affine simplex $\sigma$. This error is small if each of our simplex $\sigma \in K$ is small and does not become vanishingly thin. That means we must construct approximate triangulations with good metric and quality properties. As it turns out, assuring the simplex quality is necessary in order to have good metric properties, and we explain  in Appendix \ref{appendix: simplex_quality}.

Given a triangulation $\phi: K \to M$, we will construct a sequence of triangulations $\psi_\l: \wtilde{K}_\l \to M$ such that $\dist_M(x, y) \approx \dist_{\mathsf{PL}}(\psi_\l^{-1}(x), \psi_\l^{-1}(y))$ for every pair of points $x, y \in M$. Here $\dist_{\mathsf{PL}}$ is the piecewise-linear distance on the simplicial complex $\wtilde{K}_\l$. Using $\psi_\l$, we can transfer a Poisson point process from $M$ to the simplicial complex $\wtilde{K}_\l$, and then apply Theorem \ref{thm:univ_euclidean_general} on each flat simplex. We note that Munkres gave a similar construction and result; see \cite{munkres_differential}. His approach is based on a different subdivision  based on cubical slabs and does not provide precise bounds. Translating this result into the metric approximation required here is non-trivial. Consequently, an alternate version of his result (discovered independently) containing explicit metric distortion bounds are included here for completeness.

We use the \textit{Freudenthal-Kuhn} (FK) subdivision of a simplicial complex, which is also known as \textit{edgewise subdivision} \cite{edelsbrunner_grayson}. Unlike the barycentric subdivision, simplices in this subdivision do not become infinitely thin. Instead, there is a nonzero lower bound for their qualities, which is responsible for good metric properties. Postponing the details of the FK-triangulation construction for now, let us denote the subdivision map by $\subdiv_\l: K \rightarrow \FK_\l(K)$. A short summary of FK-subdivision is given in Subsection \ref{subsection: fkt_short}, and a full account is given in Appendix \ref{appendix: fkt_theory}.

We further apply a \textit{rectification} to eliminate first-order metric distortion left unaddressed by the FK-subdivision. Given a triangulation $\phi: K \rightarrow M$, its \textit{secant map} is defined as the unique piecewise-linear map $\phi_\triangle$ satisfying $\phi_\triangle(u_i) = \phi(u_i)$ for every simplex $\sigma \in K$ with vertices $(u_0, \ldots u_d)$. Then the \textit{rectification} of $\phi$ is defined as the composition $\psi \eqdef \phi \circ \phi_\triangle^{-1} $. Let us take the rectification of the map composed with the FK-subdivision $\phi \circ \subdiv_\l^{-1}$, and denote the resulting map $\psi_\l: \wtilde{K}_\l \to M$. This is summarized on the following commutative diagram:
\begin{center}
    \begin{tikzcd}
        K_\ell \arrow[r, "\subdiv_\ell^{-1}"] \arrow[rdr, "\phi_\l^\triangle"'] & K \arrow[r, "\phi"] & M \\
        & & \widetilde{K}_\l \arrow[u,  "\psi_\l"']
    \end{tikzcd}
\end{center}

\begin{prop} \label{prop:B}
    Let $\phi: K \to M \subset \RR^D$ be a $C^2$-triangulation, where $K$ is purely $d$-dimensional and $M$ is compact. Let $\subdiv_\l: K \to K_\l$ be the $\l$-th FK-subdivision with respect to some vertex order, and let $\psi_\l: \wtilde{K}_\l \to M$ be the rectification of $\phi \circ \subdiv_\l^{-1}$. Then for all sufficiently large $\ell$, the following assertions hold:
    \begin{enumerate}
        \item For every $x \neq y \in \wtilde{K}_\l$,
        \begin{align*}
            \bigg( 1 + \frac{\gamma}{\ell} \bigg)^{-1} \le \frac{\dist_M (\psi_\l(x), \psi_\l (y))}{\dist_{\mathsf{PL}} (x, y)}  \le \bigg( 1 + \frac{\gamma}{\ell} \bigg),
        \end{align*}
        where $\dist_{\mathsf{PL}}$ is the piecewise-linear metric on $\wtilde{K}_\l$. The constant $\gamma$ is given by:
            \[ \gamma = 3\sqrt{d} \cdot  \frac{\sup_{x \in K} \| \on{Hess}_x\phi \| }{\inf_{x \in K} \sing_{\min}(\Diff_x \phi)} \cdot \frac{\max_{\sigma \in K} \diam(\sigma)^2}{\min_{\sigma \in K} \sthick(\sigma)}. \]
        \item For every $d$-simplex $\sigma \in \wtilde{K}_\ell$, the diameter of $\sigma$ is at most $c_1 / \l$ and its thickness is at least $c_2$ regardless of $\l$. The constants $c_1,c_2$ are given in \eqref{eqn:scr_dt}, \eqref{eqn:constc012}.
    \end{enumerate}
\end{prop}

\section{Proof of the Universality}
\label{section: proofs_main}

We first prove the universality theorem for general functionals (Theorem \ref{thm:A_general}). Then we prove that the persistence functional satisfies the simplex additivity condition and the metric stability condition, which shows the universality theorem for the persistence functional (Theorem \ref{thm:A}).

\subsection{Proof of Theorem \ref{thm:A_general}}

Let $\func_n$ be a functional that satisfies conditions (1)--(7), depending on a control parameter $n$. For each $\epsilon > 0$, we will show that  for large enough $n$,
\begin{align} \label{eqn:univ_general_finalgoal}
    \bigg | \frac1n \overline{\func}_n(nf) -  \overline{\func}^* \bigg| \le \epsilon.
\end{align}

Let $\delta \in (0, 1)$ be a number satisfying $\mathcal{E}(3\delta) \le \epsilon/3$, where $\mathcal{E}$ is the error function in the metric stability condition \eqref{eqn:metric_stability} of the given functional $\func_n$. Let $\psi_\ell: \wtilde{K}_\ell \rightarrow M$ be the triangulation obtained by applying Proposition \ref{prop:B}, where we fix a number $\ell \ge \max( \gamma, 5c_1 \reach^{-1} ) \delta^{-1}$. Here $\reach = \min \{ \reach_{\phi(\sigma)} \}_{\sigma \in K}$ is the minimum of the reach of embedded simplices. Due to $\ell \ge \gamma \delta^{-1}$, Proposition \ref{prop:B} implies that for every $x, y \in \wtilde{K}$,
\begin{align*}
    (1+\delta)^{-1} \le \frac{\dist_M(\psi_\l(x),\psi_\l(y))}{\dist_{\mathsf{PL}}(x,y)} \le (1+\delta),
\end{align*}
where $\dist_{\mathsf{PL}}$ is the piecewise-linear metric on the simplicial complex $\wtilde{K}$. Due to $\ell \ge 5c_1 / \delta \reach$, Proposition \ref{prop:B} also implies that the diameter of every simplex in $\wtilde{K}_\l$ is at most $(1/5) \delta \reach$. By applying the above inequality and the assumption $\delta < 1$, the pairwise Euclidean distance between every point on $\psi_\l(\sigma)$ is at most $(2/5) \delta \reach \le (\sqrt{2}-1)\reach$. Therefore, by Lemma \ref{lem:geodesic_control}, for every pair of points $x, y \in \sigma$ lying on a single simplex $\sigma \in \wtilde{K}_\l$, we have
\begin{align*}
    (1+\delta)^{-1} \le \frac{\|\psi_\l(x)-\psi_\l(y)\|}{\dist_M(\psi_\l(x),\psi_\l(y))} \le (1+\delta).
\end{align*}
The above two bounds together imply that for every pair of points $x, y \in \sigma$, using the inequality $(1+\delta)^2 \le 1+3\delta$,
\begin{align} \label{eqn:metric_distortion_psi_pl}
    (1+3\delta)^{-1} \le \frac{\|\psi_\l (x)- \psi_\l (y)\|}{\dist_{\mathsf{PL}}(x, y)} \le (1+3\delta).
\end{align}
Define the function $\hat{f}(x) = f(\psi_\l(x)) \cdot \jacobian_x \psi_\l (x)$ where $\jacobian_x \psi$ is the Jacobian of $\psi$ at $x$, and denote the restrictions of the functions $f_\sigma = f|_{\psi_\l (\sigma)}$ and $\hat{f}_\sigma = \hat{f}|_\sigma$. By \eqref{eqn:metric_distortion_psi_pl} and the metric stability condition \eqref{eqn:metric_stability}, the following holds for every $d$-simplex $\sigma$ when $n$ is large enough:
\begin{align*}
    \frac1{n|\sigma|} \big| \overline{\func}_n(n f_\sigma) - \overline{\func}_n(n \hat{f}_\sigma; \dist_{\mathsf{PL}} ) \big| \le \mathcal{E}(3 \delta) \le \frac{\epsilon}3,
\end{align*}
where $|\sigma| = \smallint_\sigma \hat{f}(x) \> \diff x$ and $\overline{\func}_n(n \hat{f}_\sigma; \dist_{\mathsf{PL}} )$ denotes that $\dist_{\mathsf{PL}}$ is used for the functional $\func_n$ (recall \eqref{eqn:functional_metric_dependence}). The normalisation factor $1/n|\sigma|$ accounts for the fact that $n|\sigma| = \int_{\phi(\sigma)} nf_\sigma(x) \diff \hausmeas^d(x) = \int_{\sigma} n\hat f_\sigma(x) \diff x$ where $\hausmeas^d$ is the $d$-dimensional Hausdorff measure in $\RR^D$. Taking a weighted sum over $d$-dimensional simplices $\sigma \in \wtilde{K}_\l$, with weight $|\sigma|$ per simplex, yields the following:
\begin{align} \label{eqn:main_general1}
    \bigg|\frac1n \sum_{\sigma} \overline{\func}_n(n f_\sigma ) - \frac1n \sum_{\sigma} \overline{\func}_n(n \hat{f}_\sigma; \dist_{\mathsf{PL}}) \bigg| \le \frac{\epsilon}3.
\end{align}
By the simplex additivity condition \eqref{eqn:simplex_additivity1}, for large enough $n$ we have
\begin{align} \label{eqn:main_general2}
    \bigg|\frac1n \overline{\func}_n(nf) - \frac1n \sum_{\sigma} \overline{\func}_n(nf_\sigma) \bigg| \le \frac \epsilon3.
\end{align}
By the universality theorem for the Euclidean space case (Theorem \ref{thm:univ_euclidean_general}), we have
\begin{align} \label{eqn:main_general3}
    \bigg|\frac1n \sum_{\sigma} \overline{\func}_n(n \hat{f}_\sigma ; \dist_{\mathsf{PL}}) - \overline{\func}^* \bigg| \le \frac \epsilon3.
\end{align}
By \eqref{eqn:main_general1}, \eqref{eqn:main_general2}, \eqref{eqn:main_general3} and the triangle inequality, we see that \eqref{eqn:univ_general_finalgoal} holds for large enough $n$. This proves Theorem \ref{thm:A_general}.

\subsection{Proof of Theorem \ref{thm:A}} \label{subsection: persistence}

Theorem \ref{thm:A} follows from Theorem \ref{thm:A_general} by showing that the persistence functional satisfies conditions (1)--(7). The conditions (1)--(5) were shown in \cite{univ_euclidean} already, so it remains for us to show (6) and (7), which are the simplex additivity condition and metric stability condition.

\medskip
\noindent \paragraph{Simplex additivity condition.} 
We first present a few results that will be used for proving the simplex additivity condition for the persistence functional. The following lemma controls the change of persistence values under deletion of a simplex. 
\begin{lem}[Lemma 6.3, \cite{univ_euclidean}] \label{lem:delete_one_simplex}
    Let $K_\bullet$ be a filtered simplicial complex and let $\sigma$ be a simplex that enters $K_\bullet$ at some point. Then for every $\alpha$, following is true:
    \begin{align*}
        \big| \Pi_k(K_\bullet \backslash \sigma)(\alpha) - \Pi_k(K_\bullet)(\alpha) \big| \le \mathbf{1}[\dim(\sigma) \in \{k, k+1\}].
    \end{align*}
    where we abbreviate $\Pi_k(K_\bullet)(\alpha) = \big| \Pi_k(K_\bullet) \cap [\alpha, \infty) \big|$.
\end{lem}

The following statement, known as the Campbell-Mecke formula, is a key result from Palm theory, that enables us to estimate score functions over a Poisson process.

\begin{thm}[Campbell-Mecke] \label{thm:campbell_mecke}
    Let $h(x, \ptx)$ be a measurable function. Denote the $d$-dimensional Hausdorff measure in $\RR^D$ by $\hausmeas^d$. Then,
    \[
        \EE\left[\sum_{X\in \pointprocess_{nf}}h(X,\pointprocess_{nf})\right]  = n\int_M f(x) \EE[h(x,\pointprocess_{nf}) ] \diff \hausmeas^d(x),
    \]
\end{thm}

Next, we proceed to prove the simplex additivity property for the persistence functional. It is a consequence of an inequality given in terms of the radius cutoff parameter. Recall from \eqref{eqn:expected_persistence_measure} that $\overline{\Pi}_{k, \filt}(nf)$ is the expected persistence measure, and $\overline{\Pi}_{k, \filt}(nf)(B)$ is its evaluation at a Borel set $B \subseteq \RR$. For simplicity, we denote
\begin{align}
    \overline \Pi_{k, \filt}(nf)(\alpha) = \overline \Pi_{k, \filt}(nf)([\alpha, \infty)).
\end{align}

\begin{lem}[Simplex additivity for persistence] \label{prop:sigma_control}
    Let $\phi: K \rightarrow M$ be an admissible triangulation. Let $f: M \rightarrow \RR_{\ge 0}$ be a probability density, $\filt$ be either the Vietoris-Rips or \v Cech complex, and $\rmax_n$ be a radius cutoff parameter. Then there exists a constant $c>0$ so that for all integers $1 \le k \le d$, $\alpha \in (1, \infty)$ and sufficiently large $n$,
    \begin{align} \label{eqn:persistence_simplex_additivity_mainclaim}
        \bigg| \overline{\Pi}_{k, \filt}(nf)(\alpha) - \sum_{\dim\sigma = d} \overline{\Pi}_{k, \filt}(nf|_{\phi(\sigma)})(\alpha) \bigg| \le c \cdot n^{k+1} \rmax_n^{dk+1} (1 + n \rmax_n^d),
    \end{align}
    Furthermore, the persistence functional satisfies the simplex additivity \eqref{eqn:simplex_additivity1} if 
    the radius cutoff satisfies $n \rmax^d_n = o(n^{1/(d^2+d+1)})$.
\end{lem}
\begin{proof}
    Let $\ptx \subset M$ be a finite set and let $r>0$. Define the following filtered simplicial complex,
    \begin{align*}
        \wtilde{\filt}_r(\ptx) = \bigsqcup_{\substack{\sigma \in K \\ \dim \sigma = d}} \filt_r(\ptx \cap \phi(\sigma)).
    \end{align*}
    Let $U = \ball(\partial_\phi M, 2 c_\phi r)$ be the thickening of the $\phi$-skeleton of $M$ (see \eqref{eqn:phi_skeleton}), where $c_\phi = 2 / \sin(\Theta_\phi/4)$. Let $S(\ptx, r)$ be the set of all abstract simplices $\pty \subseteq \ptx$ containing a point $y \in \pty \cap U$ such that $\pty \subset \ball(y, 2r)$. We claim that the following sandwiching relations hold:
    \begin{align} \label{eqn:filt_sandwich}
        \wtilde{\filt}_r(\ptx) \>\>\subseteq\>\> \filt_r(\ptx) \>\>\subseteq\>\> \wtilde{\filt}_r(\ptx) \cup S(\ptx, r).
    \end{align}
    Indeed, let $\pty \in \filt_r(\ptx) \backslash \wtilde{\filt}_r(\ptx).$ Then its diameter is smaller than $r$ for the Vietoris-Rips complex, and $2r$ for the \v Cech complex. There also exist two points $y_1 \neq y_2 \in \pty$ such that $y_1 \in \phi(\sigma_1)$ and $y_2 \in \phi(\sigma_2)$ where $\sigma_1, \sigma_2 \in K$ are two distinct simplices. Then Proposition \ref{prop:C} implies that $y_1 \in \ball(\partial_\phi M, 2c_\phi r)$, which shows that $\pty \in S(\ptx, r)$. 
    
    Using \eqref{eqn:filt_sandwich} together with Lemma \ref{lem:delete_one_simplex}, and taking the expectation, 
    \begin{align} \label{eqn:pi_bound_twoterms}
        \bigg| \overline \Pi_{k,\filt} (nf)(\alpha) - \sum_{\dim \sigma=d} \overline \Pi_{k,\filt} (nf|_{\phi(\sigma)})(\alpha) \bigg| \le \EE\left[ S_k \right] + \EE\left[ S_{k+1} \right],
    \end{align}
    where $S_k$ is the number of all $k$-simplices in $S (\pointprocess_{nf}, \rmax_n)$. Recall that we set $U = \ball(\partial_\phi M, 2c_\phi r)$. We can bound $S_k$ using the following functional:
    \[
        h (x, \ptx) =  \mathbf{1}[x \in U] \cdot \binom{| \ptx \cap \ball(x, 2\rmax_n) | - 1}{k}.
    \]
    By the Campbell-Mecke formula (Theorem \ref{thm:campbell_mecke}),
    \[
        \EE[ S_k] \le \EE \bigg[ \sum_{X \in \pointprocess_{nf} } h (X, \pointprocess_{nf}) \bigg]
        = n \int_M f(x) \mathbf{1}[x \in U] \cdot \EE \binom{|\pointprocess_{nf} \cap \ball(x, 2\rmax_n )|}{k} \diff \hausmeas^d(x).
    \]
    Now observe that $|\pointprocess_{nf} \cap B|$ is a Poisson random variable with rate $n \int_B f(x) \diff x$ for any fixed Borel set $B$. Using the factorial moments of the Poisson distribution, we have
    \begin{align*}
        \EE[ S_k] &\le \frac {n}{k!} \int_{U\cap M} f(x) \bigg( \int_{\ball(x, 2\rmax_n) \cap M} nf(y) \diff \hausmeas^d(y) \bigg)^k \diff \hausmeas^d(x) \\
        &\le  \frac {n^{k+1} f_{\max}^{k+1}}{k!} \int_{U\cap M} \big(\hausmeas^d(\ball(x, 2\rmax_n) \cap M) \big)^k \diff \hausmeas^d(x) \\
        &\le  \frac {n^{k+1} f_{\max}^{k+1}}{k!} \hausmeas^d(U \cap M) c \omega_d^k (2\rmax_n)^{dk} \\
        &\le \bigg(\frac {f_{\max}^{k+1}}{k!} c' \omega_d^k \hausmeas^{d-1}(\partial_\phi M) \bigg)  n^{k+1} (2\rmax_n)^{dk+1}.
    \end{align*}
    for some constants $c, c' > 0$. Here we used Lemma \ref{lem:skeleton_volume} for the last inequality above. Adding the bound for $\EE[S_{k+1}]$ since \eqref{eqn:pi_bound_twoterms} involves both $\EE[S_k]$ and $\EE[S_{k+1}]$, we obtain the claim \eqref{eqn:persistence_simplex_additivity_mainclaim}. 

    For the persistence functional to satisfy the simplex additivity condition, we require $n^k \rmax_n^{dk+1} \to 0$ as $n \to \infty$ for all $1 \le k \le d+1$. Here we included the case $k=d+1$ to account for both $\EE[S_k]$ and $\EE[S_{k+1}]$ appearing in \eqref{eqn:pi_bound_twoterms}. 
    Since we assume $n\rho_n^d\to \infty$, for all $k$ we have $n^k\rho_n^{dk+1} \le n^{d+1}\rho_n^{d(d+1)+1}$. Then $(n^{d+1}\rho_n^{d(d+1)+1})^p = n^{1-q} \rho_n^d$ where $p = d/(d^2+d+1)$ and $q = 1/(d^2+d+1)$. Therefore $n^{1-q} \rmax_n^d \to 0$ holds iff $n\rho_n^d = o(n^q)$, and this assumption gives us the simplex additivity.
\end{proof}

\bigskip
\noindent \paragraph{Metric stability condition.} This condition \eqref{eqn:metric_stability} relies on two technical lemmas, whose proofs we postpone to Section \ref{section: proofs_technical}. The following lemma tells us that a perturbation of metric induces a perturbation of thresholds at which simplices are added to the Vietoris-Rips and \v Cech filtrations.

\begin{lem} \label{lem:rips_cech_interleaving}
    Suppose that $(\dist_1, \dist_2)$ are metrics on a metrisable space $M$. Assume that every closed ball of finite radius is a compact subset of $M$. Suppose that for some $\delta \ge 0$, the following inequality holds for every $x, y \in M$:
    \begin{align*}
        (1+\delta)^{-1} \dist_1(x,y) \le \dist_2(x,y) \le (1+\delta) \dist_1(x,y),
    \end{align*}
    Then for $\filt$ either the Vietoris-Rips or \v Cech filtration, the following holds for every subset $\pty \subseteq \ptx$:
    \begin{align} \label{eqn:r12_filt}
        (1+\delta)^{-1} r_{1,\filt}(\pty) \le r_{2,\filt}(\pty) \le (1+\delta) r_{1,\filt}(\pty),
    \end{align}
    where $r_{i, \filt}(\pty)$ is the threshold at which the subset $\pty \subseteq \ptx$ is added into the filtration $\filt_\bullet(\ptx, \dist_i)$.
\end{lem}

\medskip
The next lemma follows from the previous lemma. It tells us that a perturbation of metric induces a controlled change in the value of the persistence functional for both Vietoris-Rips and \v Cech filtrations.

\begin{lem} \label{lem:interleaving}
    Suppose that $(\dist_1, \dist_2)$ are metrics on a metrizable space $M$ and let $\Delta = \Delta(\dist_1, \dist_2)$ be their distortion \eqref{eqn:metric_distortion}. Let $K_\bullet^{(i)} = \filt_\bullet(\ptx; \dist_i)$, where $\filt$ is either the Vietoris-Rips or the \v Cech filtration and $\ptx \subset M$. Then the following holds for every $\alpha > (1+\Delta)^2$:
    \begin{align*}
        \bigg| \Pi_k( K_\bullet^{(1)}) (\alpha) - \Pi_k( K_\bullet^{(2)} )(\alpha) \bigg| \le \min_{i=1, 2} \bigg( \Pi_k( K_\bullet^{(i)} )((1+\Delta)^{-2} \alpha, (1+\Delta)^2 \alpha) \bigg).
    \end{align*}
\end{lem}

\medskip
We are now prepared to prove the metric stability property.

\begin{lem}[Metric stability for persistence] \label{prop:metric_perturbation_persistence}
    The persistence functional $\ptx \mapsto \Pi_k(\filt'_\bullet(\ptx))(\alpha)$ satisfies the metric stability condition \eqref{eqn:metric_stability}, if the Euclidean universal limit distribution $\overline{\Pi}_k^\univstar$ (from Theorem \ref{thm:univ_euclidean}) is continuous at $\alpha$.
\end{lem}
\begin{proof}
    Define $\mathcal{E}(t)$ as the following function:
    \begin{align} \label{eqn:metric_stability_persistence1}
        \mathcal{E}(t) = 2 \cdot \overline{\Pi}^\univstar_{k}([(1+t)^{-2} \alpha, (1+t)^2 \alpha]),
    \end{align}
    which satisfies $\lim_{t \downarrow 0} \mathcal{E}(t) = 0$ since $\overline{\Pi}_k^\univstar$ is continuous at $\alpha$. Let $\Delta(\dist, \dist_\euclidean)$ be the metric distortion \eqref{eqn:metric_distortion}. Applying the universality theorem for Euclidean space (Theorem \ref{thm:univ_euclidean}) separately for the two values $\alpha' = (1+\Delta)^{-2} \alpha$ and $\alpha' = (1+\Delta)^2 \alpha$, the following holds when $n$ is large enough:
    \begin{align} \label{eqn:metric_stability_persistence2}
        \bigg| \frac1n  \overline{\Pi} _k(nf; \dist_\euclidean)(\alpha') - \overline{\Pi}_k^\univstar(\alpha')  \bigg| \le \frac14 \mathcal{E}(\Delta) \text{, for } \alpha' = (1+\Delta)^{-2} \alpha \text{ and } (1+\Delta)^2\alpha.
    \end{align}
    Therefore, the following holds when $n$ is large enough:
\[
\begin{split}
\frac1n \bigg| \overline{\Pi}_k(nf; \dist)(\alpha) - \overline{\Pi}_k(nf; \dist_\euclidean)(\alpha) \bigg| 
        &\le  \frac1n \overline{\Pi}_k (nf; \dist_\euclidean)([ (1+\Delta)^{-2} \alpha, (1+\Delta)^2 \alpha ])  \\
        & \le \;\;\overline{\Pi}_k^\univstar ([(1+\Delta)^{-2} \alpha, (1+\Delta)^2 \alpha ])  \\
&\ \ + \bigg| \frac1n \overline{\Pi}_k (nf; \dist_\euclidean)((1+\Delta)^{-2}) - \overline{\Pi}_k^\univstar((1+\Delta)^{-2}) \bigg|   \\
&\ \ + \bigg| \frac1n \overline{\Pi}_k (nf; \dist_\euclidean)((1+\Delta)^{2}) - \overline{\Pi}_k^\univstar((1+\Delta)^{2}) \bigg|   \\
        &\le  \mathcal{E}(\Delta),
        \end{split}
    \]
where the first inequality is due to 
Lemma \ref{lem:interleaving}, the second is the triangle inequality, and the last one is due to \eqref{eqn:metric_stability_persistence1} and \eqref{eqn:metric_stability_persistence2}.
\end{proof}

\section{Proofs of geometric results}
\label{section: proofs_geometric}

In this section we prove Propositions \ref{prop:C} and \ref{prop:B}, which comprise the main geometric tools to transfer functionals from a Euclidean space to a triangulable space.

\subsection{Interference control and Proposition \ref{prop:C}}
\label{subsection: interference_control}

We set up some notations first. Given a $C^2$ embedding $\phi: \sigma \to \RR^D$, denote the embedded simplex as $\tilde \sigma = \phi(\sigma)$. For each point $x \in \partial \tilde \sigma$, we define the \textit{inward cone} $\incone_x \tilde \sigma$, the \textit{normal inward cone} $\nincone_x \tilde \sigma$, and the \textit{unit normal inward cone} $\nincone^1_x \tilde \sigma$:
\begin{align*}
    \incone_x \tilde \sigma =& \big\{ v \in \RR^{D} \>|\> \exists \epsilon>0: \> x + \epsilon v \in \sigma^\circ \big\} \\
    \nincone_x \tilde \sigma =& \incone_x \tilde \sigma \cap \normspc_x \tilde \sigma_x  \\
    \nincone^1_x \tilde \sigma =& \big\{ v \in \nincone_x \tilde \sigma \>|\> \|v\|=1 \big\}
\end{align*}
where $\tilde \sigma_x$ is the smallest facet of $\sigma$ whose interior contains $x$, and $\normspc_x M$ is the normal space at a manifold $M$.

We introduce three lemmas that will be used for the proof of Proposition \ref{prop:C}, whose proofs we postpone to Section \ref{section: proofs_technical}. The first lemma is the decomposition of the inward cone into the normal inward cone and the tangent space:

\begin{lem}\label{lem:inward_cone_decomp}
    Let $\phi: \sigma \rightarrow \RR^D$ be a $C^2$ embedding of an affine simplex $\sigma$, and denote $\tilde \sigma = \phi(\sigma)$. Then for each $x \in \partial \tilde \sigma$, the following map $g$ is a homeomorphism:
    \begin{align*}
        g: \nincone_x \tilde \sigma \times \tanspc_x \tilde \sigma \to & \>\> \incone_x \tilde \sigma, \qquad g(v, w) = v+w
    \end{align*}
\end{lem}

The next lemma is an exercise in calculus.

\begin{lem} \label{lem:variational1}
    Let $\theta, \epsilon, \ell, \reach$ be fixed real numbers satisfying $\theta \in (0, \pi]$, $\epsilon \in [0, \sin(\theta/4)]$, and $\ell \in [0, \epsilon \reach]$. Define the following function $g(\lambda)$:
    \begin{align*}
        g(\lambda) = \sqrt{\ell^2 + \lambda^2 - 2\ell \lambda \cos \theta} - \epsilon (\ell + \lambda)
    \end{align*}
    Then $g(\lambda)$ has the following lower bound:
    \begin{align*}
        \min_{0 \le \lambda \le \epsilon \reach} g(\lambda) \ge \ell (\sin \theta_+ - 3\epsilon) \text{, where } \theta_+ = \min\bigg( \theta, \frac {3\pi}4 \bigg)
    \end{align*}
\end{lem}

Using Lemmas \ref{lem:inward_cone_decomp} and \ref{lem:variational1}, the following lemma is proven. This is the key technical result required to prove Proposition \ref{prop:C}.

\begin{lem}\label{lemma: interference1}
    Let $\tilde \sigma_1, \tilde \sigma_2 \subset \RR^D$ be two embedded $C^2$ simplices of reach $\reach$, and assume that they share a face. Let $z \in \tilde \sigma_1 \cap \tilde \sigma_2$ and let following be unit vectors: $u \in \nincone^1_z(\tilde \sigma_1)$, $v_\normspc \in \nincone_z^1(\tilde \sigma_2)$, $v_\tanspc \in \tanspc_z(\tilde \sigma_2)$. Assume that the angle $\theta = \angle(u, v_\normspc)$ satisfies $\theta \in (0, \pi]$. Let $\epsilon, \ell, \lambda_\normspc, \lambda_\tanspc$ be numbers in the following range: $\epsilon \in (0, \sin(\theta/4)]$ and $\ell, \lambda_\normspc, \lambda_\tanspc \in [0, \min(1/5, \epsilon) \cdot \reach]$. Then the following holds:
    \begin{align*}
        \|x-y \| \ge \ell(\sin\theta_+ - 3\epsilon)
    \end{align*}
    where $x, y, \theta_+$ are defined as follows:
    \[ x = \exp_z^{\sigma_1}(\ell u), \quad y = \exp_z^{\sigma_2}(\lambda_\normspc v_\normspc + \lambda_\tanspc v_\tanspc), \quad \theta_+ = \min \bigg( \theta, \frac{3\pi}4 \bigg) \]
\end{lem}

\medskip
Lastly, we recall the Riemannian tubular neighbourhood theorem.

\begin{thm}[Riemannian tubular neighbourhood]
    Let $M$ be a Riemannian manifold. Every embedded submanifold has a tubular neighbourhood.
\end{thm}

\medskip
We are now ready to prove Proposition \ref{prop:C}.

\bigskip
  \noindent \underline{\textbf{Proof of Proposition \ref{prop:C}.}}

We will prove a slightly stronger result, where the constant $2 / \sin(\Theta_\phi/4)$ is replaced by $(1+\epsilon) / \sin(\Theta_\phi/4)$. Fix $\epsilon \in (0, 1]$. We claim that there exists $R_\epsilon > 0$ such that for every $r \in [0, R_\epsilon ]$ and every pair of simplices $(\sigma_1, \sigma_2)$ in $M$, the following holds:
\begin{align*}
    \ball(\tilde \sigma_2, r) \cap \tilde \sigma_1 \subseteq \ball \bigg( \partial_\phi M, \frac{(1+\epsilon) r}{\sin(\Theta_\phi/4)} \bigg) 
\end{align*}
Here $R_\epsilon$ is given by:
\begin{align*}
    R_\epsilon = R (\epsilon \rho_0), \quad \rho_0 = \min\bigg\{ \rho_{\on{tube}}, \> \frac{\reach }6 \sin \frac{\Theta_\phi }4 \bigg\}
\end{align*}
where $\rho_{\on{tube}}$ is the radius of tubular neighbourhood of $\partial_\phi M$ and $R$ is a function defined in \eqref{eqn:R_param}. We will explain the proof in the following steps.

  \noindent \paragraph{Step 1: Setup.} Define $R(s)$ as the following parameter depending on $s \ge 0$:
\begin{align} \label{eqn:R_param}
    R(s) = \min_{\sigma_1, \sigma_2 \in K} \dist_\euclidean \big( \tilde \sigma_1 \backslash \mathcal{N}(s), \tilde \sigma_2 \big)
\end{align}
where $\mathcal{N}(s)$ is the geodesic normal tube of radius $s$ centred at $\tilde \sigma_1 \cap \tilde \sigma_2$, and geodesics are on $\tilde \sigma_1$. In the definition of $R(s)$, $(\sigma_1, \sigma_2)$ are taken over $d$-simplices of $\sigma_1, \sigma_2 \in K$ with a nonempty intersection. If $x \in \sigma_1, y \in \sigma_2$ with $\|x-y\| < R(s)$, then $x, y$ are both within distance $s$ from $\sigma_1 \cap \sigma_2$.

Let $\reach > 0$ be the simplex-wise reach of $M$. Define $\rho$ and $\epsilon'$ as follows:
\begin{align*}
    \rho = \frac14 \min \big( \rho_{\on{tube}}, \> \epsilon' \reach \big), \quad \epsilon' = \frac{\epsilon}{6}\cdot \sin \frac{\Theta_\phi}4
\end{align*}
We note the following simple bounds, which follow from $\epsilon' \le 1/6$ and $3 \epsilon' \le \frac{\epsilon}{1+\epsilon} \sin (\Theta_\phi/4)$.
\begin{align} \label{eqn:interference_main5}
    \rho \le \frac{\reach}{20}, \quad
    \frac1{\sin (\Theta_\phi/4) - 3\epsilon'} \le \frac{1+\epsilon}{\sin (\Theta_\phi/4)}
\end{align}

\medskip
  \noindent \paragraph{Step 2: Interference lemma.} Let $x \in \tilde \sigma_1$ be a point such that $\dist_\euclidean(x, \tilde \sigma_2) \le R(\rho)$. By the definition of $R(\rho)$, there exists a point $z \in \tilde \sigma_1 \cap \tilde \sigma_2$, a geodesic $\gamma(z, x) \subset \tilde \sigma_1$ connecting $(z, x)$ that is normal to $\tilde \sigma_1 \cap \tilde \sigma_2$, and the length $\ell$ of $\gamma(z, x)$ satisfies $\ell \le \rho$. Since the length of a geodesic is at least the ambient Euclidean distance,
\begin{align} \label{eqn:interference_main3}
    \ell \ge \|x-z\| \ge \dist_\euclidean(x, \tilde \sigma_1 \cap \tilde \sigma_2)
\end{align}

Consider a point $y \in \tilde \sigma_2 \cap \ball(z, 3\rho)$, and let $\gamma(z, y)$ be a geodesic connecting $(z, y)$. Let $v$ and $\lambda$ be the initial velocity and length of $\gamma (z, y)$, respectively. By the inward cone decomposition (Lemma \ref{lem:inward_cone_decomp}), there exist unit vectors $v_\normspc \in \nincone_z^1(\tilde \sigma_2)$, $v_\tanspc \in \tanspc_z(\tilde \sigma_2)$, and numbers $\lambda_\normspc, \lambda_\tanspc \ge 0$ so that $v = \lambda_\tanspc v_\tanspc + \lambda_\normspc v_\normspc$. Define also $u \in \nincone_z^1(\tilde \sigma_1)$ to be the initial velocity of $\gamma(z, x)$. We apply geodesic control\footnote{Condition of the lemma is met in both cases since $\|x-z\| \le \ell \le \rho \le 0.4 \reach$ and $\|y-z\| \le 3 \rho \le 0.4 \reach$. Here we used $\rho \le \reach/24$.} (Lemma \ref{lem:geodesic_control}) to obtain the following:
\begin{align*}
    \ell \le & \rho + \frac{\rho^2}{\reach} \le 4\rho \le \min(1/5, \epsilon') \cdot \reach \\
    \lambda \le & 3\rho + \frac{9\rho^2}{\reach} \le 4 \rho \le \min(1/5, \epsilon') \cdot \reach
\end{align*}

We can then bound $\|x-y\|$ with the interference lemma. Define $\theta = \angle(u, v_\normspc)$, which satisfies $\theta \ge \Theta_\phi$ since $u \in \nincone_z^1(\tilde \sigma_1)$ and $v_\normspc \in \nincone_z^1(\tilde \sigma_2)$. Since $\epsilon' \le \sin(\Theta_\phi/4) \le \sin(\theta/4)$ and\footnote{Note that $\max(\lambda_\normspc, \lambda_\tanspc) \le \lambda$ since $\lambda = \sqrt{\lambda_\normspc^2 + \lambda_\tanspc^2}$.} $\ell, \lambda_\normspc, \lambda_\tanspc \in [0, \min(1/5, \epsilon' ) \cdot \reach]$, we meet the condition of the interference lemma (Lemma \ref{lemma: interference1}), and obtain:
\begin{align*}
    \|x-y\| \ge \ell \cdot (\sin \theta_+ - 3 \epsilon') \text{, where } \theta_+ = \min(\theta, (3/4)\pi)
\end{align*}
Since the function $f(t) = \sin (\min(t, (3/4)\pi))$ is not a monotone function over $t \in [0, \pi]$, we will use a weaker monotone bound. Since $\sin (\min(t, (3/4)\pi)) \ge \sin(t/4)$ over $t \in [0, \pi]$, the above equation along with \eqref{eqn:interference_main3} implies the following weaker inequality:
\begin{align} \label{eqn:interference_main4}
    \|x-y\| \ge \dist_\euclidean(x, \tilde \sigma_1 \cap \tilde \sigma_2) \cdot (\sin (\theta/4) - 3 \epsilon')
\end{align}

\medskip
  \noindent \paragraph{Step 3: Combining everything.} Denote $U = \ball(z, 3\rho)$. Consider \eqref{eqn:interference_main4} over all $y \in \tilde \sigma_2 \cap U$ and take the infimum of both sides. Then since $\theta \ge \Theta_\phi$,
\begin{align} \label{eqn:interference_main1}
    \dist_\euclidean(x, \tilde \sigma_2) = \dist_\euclidean(x, \tilde \sigma_2 \cap U) \ge \dist_\euclidean(x, \tilde \sigma_1 \cap \tilde \sigma_2) \cdot (\sin (\Theta_\phi/4) - 3\epsilon')
\end{align}
The equality above is due to following triangle inequality argument. We have $\dist_\euclidean(x, \tilde \sigma_2 \backslash U) \ge 2\rho$, since for each $y' \in \tilde \sigma_2 \backslash U$, it holds that $\|x-y'\| \ge \|y'-z\| - \|x-z\| \ge 3\rho - \rho \ge 2 \rho$. And also $\dist_\euclidean(x, \tilde \sigma_2 \cap U) \le \|x-z\| \le \rho$, which implies $\dist_\euclidean(x, \tilde \sigma_2 \cap U ) < \dist_\euclidean(x, \tilde \sigma_2 \backslash U )$ and therefore $\dist_\euclidean(x, \tilde \sigma_2) = \dist_\euclidean(x, \tilde \sigma_2 \cap U)$.
By Equations \eqref{eqn:interference_main1} and \eqref{eqn:interference_main5} we are essentially done:
\begin{align*}
    \dist_\euclidean(x, \tilde \sigma_1 \cap \tilde \sigma_2) \le \frac{\dist_\euclidean(x, \tilde \sigma_2)}{\sin (\Theta_\phi/4) - 3 \epsilon'} \le \frac{1+\epsilon}{\sin (\Theta_\phi/4)} \cdot \dist_\euclidean(x, \tilde \sigma_2)
\end{align*}
Going back to the original claim, consider a number $r \in [0, R_\epsilon ]$ and a point $x' \in \ball(\tilde \sigma_2, r) \cap \tilde \sigma_1$. It is easy to see that $R_\epsilon \le R(\rho)$. Plugging in $x'$ to the $x$ above gives $x' \in \ball(\tilde \sigma_1 \cap \tilde \sigma_2, \frac{1+\epsilon}{\sin(\Theta_\phi/4)} \cdot r)$. This gives the original main claim:
\begin{align*}
    \ball(\tilde \sigma_2, r) \cap \tilde \sigma_1 \subseteq \ball \bigg( \tilde \sigma_1 \cap \tilde \sigma_2, \frac{(1+\epsilon)r }{\sin(\Theta_\phi/4)} \bigg) \subseteq \ball \bigg( \partial_\phi M, \frac{(1+\epsilon) r}{\sin(\Theta_\phi/4)} \bigg)
\end{align*}

On top of Proposition \ref{prop:C}, we need the assurance that subdivision of simplicial complexes will ensure that conclusions of Proposition \ref{prop:C} still hold:

\begin{prop}[Safe subdivision]
    Let $\phi: K \rightarrow M$ be a purely $d$-dimensional $C^2$ triangulation of a compact set $M \subset \RR^D$. Let $\eta: K \rightarrow L$ be a subdivision of simplicial complexes and let $\psi = \phi \circ \eta^{-1} : L \rightarrow M$ be an induced triangulation. Assume that $\Theta_{\phi} > 0$, and denote $c_\phi = 2/\sin(\Theta_\phi/4)$. Then there exists $r_0 > 0$ such that the following holds for all $r \in [0, r_0]$:
    \begin{align*}
        \bigcup_{\substack{\tau_i \neq \tau_j \in L \\ \dim \tau_i = \dim \tau_j = d}}\ball(\tilde\tau_i, r) \cap \tilde \tau_j \subseteq \ball(\partial_{\psi} M, c_{\phi} r)
    \end{align*}
\end{prop}
\begin{proof}
    We have two cases for $(\tau_i, \tau_j)$. Before proceeding, let $r_0 = \min(R_1, \min_{\sigma \in K} (\reach_{\tilde \sigma}))$, where $R_1$ is the radius constant appearing in the beginning of the proof of Proposition \ref{prop:C}.

    \vspace{2mm}
   \noindent \paragraph{Case 1.} Firstly suppose that $\tau_i, \tau_j \subset \sigma$ are lying on the same simplex of $K$, before subdivision. Now suppose that $x \in \ball(\tilde \tau_i, r) \cap \tilde \tau_j$, which implies that there is a point $y \in \tilde \tau_i$ so that $\dist_\euclidean(x, y) \le r$. Then there exists a continuous path $\gamma: [0, T] \rightarrow M$ so that $\gamma(0) = x$ and $\gamma(T) = y$ which satisfies $\dist_\euclidean(x, \gamma(t)) \le \dist_\euclidean(x, y)$ for all $t \in [0, T]$. 

    The following is one way to construct such a path $\gamma$, which is Lemma 2.1.6. of \cite{uzu_thesis}. Alternatively, one can use a geodesic for this purpose. Define $\bar\gamma(t) = (1-t)x + t y$, and define $\gamma(t) \in \tilde \sigma$ be the unique nearest-point projection from $\bar \gamma(t)$ to $\tilde \sigma$. The uniqueness of this nearest point is guaranteed by the assumption that $r \le r_0 \le \reach_{\tilde \sigma}$. Observe that $\| x-\gamma(t) \| \le \|x-\bar\gamma(t)\| + \|\bar \gamma(t) - \gamma(t)\| \le \|x-\bar\gamma(t)\| + \|y - \bar \gamma(t)\| = \|x-y\|$, which is the desired property of $\gamma$. Continuity of the nearest-point projection map is proven in Lemma 2.1.5. of \cite{uzu_thesis}.
    
    Equipped with the path $\gamma$, let $t_0$ be the first time $\gamma$ touches the boundary $\partial \tilde \tau_j$, i.e. it is the supremum of all times $t$ for which $\gamma([0, t]) \subset \tilde \tau_j$. Then:
    \begin{align*}
        \dist_\euclidean(x, \partial_\psi M) \le \dist_\euclidean(x, \gamma(t_0)) \le \dist_\euclidean(x, y) \le r
    \end{align*}
    The radius constant is $1$ for this case, and since $c_\phi = 2 / \sin(\Theta_\phi / 4) \ge 1$, we can use the same constant $c_\phi$ to account for Case 1.
    
    \vspace{2mm}
      \noindent\paragraph{Case 2.} Next suppose that $\tau_i \subset \sigma_i$ and $\tau_j \subset \sigma_j$, with $\sigma_i \neq \sigma_j \in K$. This case is already accounted for by Proposition \ref{prop:C}, because:
    \begin{align*}
        \big( \ball(\tilde \tau_i, r) \cap \tilde \tau_j \big) \subseteq \big( \ball(\tilde \sigma_i, r) \cap \tilde \sigma_j \big) \subseteq \ball(\partial_\phi M, c_\phi r) \subseteq \ball(\partial_\psi M, c_\phi r)
    \end{align*}
    This accounts for both cases and finishes the proof.
\end{proof}

\subsection{Rectification and Proposition \ref{prop:B}}
\label{subsection: fkt_short}

The following is a short summary of the FK-subdivision method of a simplicial complex; we give a self-contained derivation of it in Appendix \ref{appendix: fkt_theory}. The standard order simplex is defined as $\ordsimp{d} = \{(t_1, \ldots t_d)\>|\> 1 \ge t_1 \ge \cdots \ge t_d \ge 0\}$. The set of all simplices $\vecbf{n} + \permelt \ordsimp{d}$, where $\vecbf n \in \ZZ^d$ and $\permelt$ ranges over coordinate permutations, forms a triangulation $\FK(\RR^d)$ of the Euclidean space $\RR^d$. A standard fact is that the radially scaled simplices $\l \cdot \ordsimp{d} \subset \RR^d$ are sub-complexes of $\FK(\RR^d)$. Therefore, we obtain the subdivision $\FK_\l(\sigma)$ of a $d$-simplex $\sigma$ by pushforwarding $\FK(\RR^d)|_{\l \cdot \ordsimp{d}}$ to $\sigma$ by matching their vertices. Here, the order of vertices of $\sigma$ matters, and we explain this in Appendix \ref{appendix: fkt_vertex_ordering}. Finally, when we are given an affine simplicial complex $K$ along with an ordering of its vertices, we obtain $\FK_\l(K)$ by applying the FK-subdivision per every simplex of $K$.

At the end of this subsection, Proposition \ref{prop:B} is proven by rectifying FK-subdivisions of a given triangulation. We will need the following second-order control:

\begin{lem} \label{lem:calculus1}
    Let $\phi: U \rightarrow \RR^D$ be a $C^2$ map, where $U \subseteq \RR^d$ is a path-connected open subset. Then the following holds for all $x, y \in U$ and $v \in \RR^d$:
    \begin{align*}
        \|\phi(y) - \phi(x) - \Diff_x \phi(y-x) \| \le & \frac {h_\phi}2 \cdot \|y-x\|^2 \\
        \|\Diff_y \phi(v) - \Diff_x \phi(v) \| \le & h_\phi \cdot \|y-x \| \cdot \|v\|
    \end{align*}
    where $h_\phi$ is the supremum of the Hessian operator norm\footnote{Precise definition is $\big\| \hess_x(\phi) \big\|_\mathsf{op}^2 \eqdef
    \sup \bigg\{ \sum_{i=1}^D \big( u^\top \hess_x(\phi_i) v \big)^2 \>\bigg|\> \|u\|=\|v\|=1 \bigg\}$.}, $h_\phi \eqdef \sup_{x \in U} \big\| \hess_x(\phi) \big\|_\mathsf{op}$. 
\end{lem}

\medskip
The following is the key calculation needed to prove Proposition \ref{prop:B}.

\begin{prop}[Metric distortion of rectified triangulation] 
\label{prop:metric_distortion_rectification}
    Let $\phi: K \rightarrow M$ be a $C^2$ triangulation of a compact set $M \subset \RR^D$ and let $\psi: \widetilde{K} \rightarrow M$ be its rectification. Then for every pair of points $x, y \in \widetilde{K}$, the following holds if $\alpha\beta < 1$:
    \begin{align} \label{eqn:prop_distortion1}
        \frac{\dist_{M} (\psi(x), \psi(y) )}{\dist_{\widetilde{K}} (x, y)} \in \bigg[ \frac{1-\alpha \beta}{1+\alpha \beta}, \>\> \frac{1+\alpha \beta}{1-\alpha \beta} \bigg]
    \end{align}
    where $\alpha, \beta$ are given as follows\footnote{Note that $\beta$ depends on simplices of $K$, not simplices of $\widetilde K$.}:
    \begin{align*}
        \alpha = \frac{\sup_{x \in K} \| \hess_x \phi \|_{\mathsf{op}}}{\inf_{x \in K} \sing_{\min}(\Diff_x \phi)}, \quad \beta = \sqrt{d}\cdot \max_{\sigma \in K}  \frac{\diam(\sigma)^2}{\sthick(\sigma)}
    \end{align*}
\end{prop}
\begin{proof} The proof consists of multiple steps. \\
    \textbf{\underline{Reduction to \eqref{eqn:metric_distortion_6}.}} The main claim \eqref{eqn:prop_distortion1} will be implied by the following:
    \begin{align} \label{eqn:metric_distortion_6}
        \frac{\| \Diff_x \phi(v) \|}{\| \Diff_x \phi_\triangle(v) \|} \in \bigg[ \frac{1 - \alpha \beta}{1+\alpha \beta}, \> \frac{1+\alpha \beta}{1- \alpha \beta} \bigg]
    \end{align}
    If we assume the above, the following holds for every $y \in \wtilde{K}$ and $u \in \tanspc_{x'} \wtilde{K}$ (here denote $x = \phi_\triangle^{-1}(y)$) :
    \begin{align*}
        \| \Diff_{y} \psi (u) \| = \big\| (\Diff_{x} \phi) ( (\Diff_{x} \phi_\triangle)^{-1} (u) ) \big\| \le \frac{1+\alpha \beta}{1-\alpha \beta} \big\| (\Diff_{x} \phi_\triangle) ( (\Diff_{x} \phi_\triangle)^{-1} (u) ) \big\| = \frac{1+\alpha \beta}{1-\alpha \beta} \|u\|
    \end{align*}
    This implies the following for all piecewise $C^1$ curves $\gamma: [0,T] \rightarrow \wtilde{K}$:
    \begin{align*}
        L(\psi \circ \gamma) = \int_0^T \| (\psi \circ \gamma)'(t) \| \diff t = \int_0^T \| \Diff_{\gamma(t)} \psi (\gamma'(t)) \| \diff t \le \int_0^T \frac{1+\alpha \beta}{1-\alpha \beta} \|\gamma'(t) \| \diff t = \frac{1+\alpha \beta}{1-\alpha \beta} L(\gamma)
    \end{align*}
    This implies \eqref{eqn:prop_distortion1} by taking infimum\footnote{The other inequality in \eqref{eqn:prop_distortion1} is implied by symmetry. Apply the above logic to $\psi^{-1} = \phi_\triangle \circ \phi^{-1}$ instead of $\psi$.}:
    \begin{align*}
        \dist_M( \psi(x), \psi(y)) = \inf \bigg\{ L(\psi \circ \gamma) \>\bigg|\> \gamma \text{ connects } (x, y) \bigg\} \le \frac{1+\alpha \beta}{1-\alpha \beta} \dist_{\wtilde{K}}(x, y)
    \end{align*}
    \textbf{\underline{Proof of \eqref{eqn:metric_distortion_6}.}} Suppose that $x \in \sigma$ where $\sigma \in K$ is a $d$-dimensional affine simplex. Since $\sigma$ is an affine simplex, we may assume that $\sigma \subset \RR^d$. Let the vertices of a simplex $\sigma \in K$ be $u_0, \ldots u_d \in \RR^d$ and let the centroid be $\bar{u} = \frac1{d+1}(u_0 + \cdots + u_d)$. The derivative $\phi_\triangle'$ of the affine linear map $\phi_\triangle$ is the same everywhere, given by linearly extending the following relation over all $(i, j)$:
    \begin{align*}
        \Diff_x \phi_\triangle(u_i - u_j) =& \phi(u_i) - \phi(u_j)
    \end{align*}
    Define the following:
    \begin{align*}
        z_0 =& \Diff_{\bar{u}} \phi(v), 
        \quad z_1 = \Diff_x \phi_\triangle(v) - \Diff_{\bar{u}} \phi(v),
        \quad z_2 = \Diff_{x} \phi(v) - \Diff_{\bar{u}} \phi(v)
    \end{align*}
    Then the following holds:
    \begin{align} \label{eqn:metric_distortion_3}
        & \frac{\| \Diff_{x} \phi_\triangle (v) \|}{\| \Diff_{x} \phi(v) \|}
        = \frac{\| z_1 + z_0 \|}{\| z_2 + z_0 \|}
        \in \bigg[ \frac{1 - \| z_1 \| / \| z_0 \| } {1 + \| z_2 \| / \| z_0 \|},
        \frac{1 + \| z_1 \| / \| z_0 \| } {1 - \| z_2 \| / \| z_0 \| } \bigg]
    \end{align}
    Thus we will obtain upper bounds for $\| z_1 \| , \| z_2 \| $ and a lower bound for $\| z_0 \|$. A lower bound of $\| z_0 \|$ is found by using the minimum singular value:
    \begin{align} \label{eqn:metric_distortion_1}
        \|\Diff_{\bar{u}} \phi(v) \| \ge \sing_{\min}(\Diff_{\bar{u}} \phi) \cdot \|v \|
    \end{align}
    \textbf{\underline{Upper bound of $\|z_1\|, \|z_2\|$.}} To find an upper bound of $\|z_1\|, \|z_2\|$, rewrite $v$ in a basis:
    \begin{align} \label{eqn:metric_distortion_2}
        v =& \sum_{i=1}^d a_i (u_i - u_0) = A_{\sigma, 0} \cdot \vecbf a
    \end{align}
    where $A_{\sigma, j} = [u_0 - u_j, \ldots u_d - u_j] \in \RR^{D \times d}$ is the matrix of edge vectors and $\vecbf a$ is a vector whose entries are $\vecbf a = (a_1, \ldots a_d) \in \RR^d$. Therefore,
    \begin{align*}
        z_1 =& \sum_{i=1}^d a_i \cdot \big( \phi(u_i) - \phi(u_0) - \Diff_{\bar{u}} \phi (u_i - u_0) \big) \\
        =& \sum_{i=1}^d a_i \cdot \big( (\phi( u_i) - \phi(\bar{u}) - \Diff_{\bar{u}} \phi (u_i - \bar{u})) - (\phi(u_0) - \phi(\bar{u}) - \Diff_{\bar{u}} \phi (u_0 - \bar{u})) \big) \\
        z_2 =& \sum_{i=1}^d a_i \cdot \big( \Diff_{x} \phi(u_i - u_0) - \Diff_{\bar{u}} \phi(u_i - u_0) \big)
    \end{align*}
    By Lemma \ref{lem:calculus1}, we bound the norm of each summand as follows (this is where the $C^2$ assumption on the triangulation $\phi$ is used):
    \begin{align*}
        \|\phi(u_i) - \phi(\bar{u}) - \Diff_{u} \phi(u_i - \bar{u} ) \| \le& \frac {h_\phi} 2 \cdot \|u_i - \bar{u} \|^2 \le \frac {h_\phi} 2 \cdot \diam(\sigma)^2 \\
        \| \Diff_{x}\phi (u_i - u_0) - \Diff_{\bar{u}} \phi(u_i - u_0) \| \le& h_\phi \cdot \| x - \bar{u}\| \cdot \|u_i - u_0\| \le h_\phi \cdot \diam(\sigma)^2
    \end{align*}
    where $h_\phi = \sup_{y \in K} \|\hess_y \phi \|$. Therefore,
    \begin{align} \label{eqn:metric_distortion_5}
        \|z_i\| \le \|\vecbf a\|_1 \cdot h_\phi \cdot \diam(\sigma)^2 \quad \text{ (for $i=1, 2$)}
    \end{align}
    Equipped with the lower bound for $\|z_0\|$ \eqref{eqn:metric_distortion_1} and the upper bound for $\|z_i\|$ \eqref{eqn:metric_distortion_5}, we then have the following (for $i=1,2$):
    \begin{align} \label{eqn:metric_distortion_4}
        \frac{\|z_i \| }{\| z_0 \|} \le \frac{\|\vecbf a\|_1 \cdot h_\phi \cdot \diam(\sigma)^2 }{\sing_{\min}(\Diff_{\bar u} \phi) \cdot \| A_{\sigma, 0} \cdot \vecbf a \| } \le \frac{h_\phi \cdot \diam(\sigma)^2}{\inf_{u' \in K} \sing_{\min}(\Diff_{u'} \phi)} \cdot \frac{\|\vecbf a\|_1}{\| A_{\sigma, 0} \cdot \vecbf a\|_2}
    \end{align}
    We can further tighten the above bound by replacing $A_{\sigma, 0}$ appearing in \eqref{eqn:metric_distortion_2} by $A_{\sigma, 1}, \ldots A_{\sigma, d}$ and taking the tightest bound possible, and this gives us the expression for S-width \eqref{eqn:simplex_quality}. Upon doing so, we exactly obtain the expression of $\alpha \beta$ stated in the proposition statement, giving the following (for $i=1,2$):
    \begin{align*}
        \frac{\|z_i\|}{\|z_0\|} \le \frac{h_\phi}{\inf_{u' \in K} \sing_{\min}(\Diff_{u'} \phi)} \cdot \max_{\sigma \in K} \frac{\diam(\sigma)^2}{\sthick(\sigma)}
    \end{align*}
    Plugging this into \eqref{eqn:metric_distortion_3} gives \eqref{eqn:metric_distortion_6}, finishing the proof of \eqref{eqn:prop_distortion1}.
\end{proof}

\medskip
\noindent \underline{Proof of (1) of Proposition \ref{prop:B}}.

Let $K_\l$ be the $\l$-th FK-subdivision of $K$ with respect to some vertex ordering, and let $\subdiv_\ell: K \rightarrow K_\ell$ be the subdivision map. Apply Proposition \ref{prop:metric_distortion_rectification} to the map $\phi_\ell = \phi \circ \subdiv_\ell^{-1}$. Since $\subdiv_\ell$ is an isometry, it does not affect local metric properties when we switch from $\phi$ to $\phi_\ell$. Let $\alpha_\l, \beta_\l$ be the parameters $\alpha, \beta$ in Proposition \ref{prop:metric_distortion_rectification} when we use $\phi_\l$ for $\phi$. Recall also that $\diam(\lambda \cdot \sigma) = \lambda \cdot \diam(\sigma)$ and $\sthick(\lambda \cdot \sigma) = \lambda \cdot \sthick(\sigma)$. This gives the following expression for $\gamma$ appearing in Proposition \ref{prop:B}:
    \begin{align*}
        & \alpha_\ell = \frac{\sup_{x \in K_\ell} \| \hess (\phi_\ell) \|}{\inf_{x \in K_\ell} \sing_{\min}(\Diff_x \phi_\ell)} = \frac{\sup_{x \in K} \| \hess (\phi) \|}{\inf_{x \in K} \sing_{\min}(\Diff_x \phi)} = \alpha_0 \\
        & \beta_\ell = \sqrt{d}\max_{\sigma \in K_\ell} \frac{\diam(\sigma)^2}{\sthick(\sigma)} = \sqrt{d}\max_{\sigma \in K} \frac{\diam(\sigma)^2 / \ell}{\sthick(\sigma)} = \frac{\beta_0}{\ell} \\
        &\qquad\qquad\implies \alpha_\l \beta_\l = \frac{\alpha_0 \beta_0}{\l} = \frac{\gamma}{\l} \text{, where } \gamma = \alpha_0 \beta_0
    \end{align*}

\medskip
\noindent \underline{Proof of (2) of Proposition \ref{prop:B}.} 

    As before, let $K_\ell$ be the $\ell$-th FK-triangulation of $K$ with respect to some vertex ordering $v_\bullet^K$. Fix an affine embedding $\sigma_0 \hookrightarrow \RR^d$ for every $d$-dim. simplex $\sigma_0 \in K$. Let $\eta: \RR^d \rightarrow \RR^d$ be an affine-linear isomorphism defined by $\eta(\vecbf e_i) = v_i^0 - v_{i-1}^0$, where $(v_0^0, \ldots v_d^0)$ are vertices of $\sigma_0$ (in the inherited vertex order $v_\bullet^{0} = v_\bullet^K|_{\sigma_0}$). This makes it so that $\eta(\vecbf e_0 + \cdots + \vecbf e_i) = v_i^0$ and $\eta (\ordsimp{d}) = \sigma_0$.
    
    Let $Q(\sigma_0)$ be the isometry classes of all $d$-dimensional simplices in the pushforwarded triangulation $\eta(\FK(\RR^d))$. By Lemma \ref{lem:fkt_isometry_bound}, $Q(\sigma_0)$ is a finite set of cardinality at most $d!$. Diameters of its elements are bounded as follows:
    \begin{align*}
        \max_{\tau \in Q(\sigma_0)} \diam(\tau) &= \max \big\{ \| \eta(s_1, \ldots s_d) \| \>\big|\> s_i \in \{-1,0,+1\} \big\} \\
        &= \max \big\{ \| \textstyle{\sum}_{i=1}^d s_i (v_i^0 - v_{i-1}^0) \| \>\big|\> s_i \in \{-1,0,+1\} \big\} \\
        &\le  \;d \cdot \diam(\sigma_0)
    \end{align*}
    Now given a simplex in the subdivision $\sigma \in K_\ell$, it is isometric to the rescaled simplex $\ell^{-1} \cdot \tau$ for some $\tau \in Q(\sigma_0)$. i.e. the following holds for some rotation $R \in O(d)$:
    \begin{align*}
        & \sigma = x + \l^{-1} \tau^R \subset \RR^d, \quad \text{where } \tau^R = R(\tau) \\
        \implies & u_i = x + \l^{-1} v_i^R \in \RR^d, \quad \text{ where } v_i^R = R(v_i)
    \end{align*}
    Define the following approximations $w_i \approx \phi(u_i)$ and $\hat \sigma, \bar \sigma \approx \phi(\sigma)$ (they all lie on $\RR^D$, the codomain of $\phi$):
    \begin{align*}
        w_i &= \phi(x) + \l^{-1} \cdot \Diff_{x} \phi(v_i^R) \\
        \bar \sigma &= \mathsf{ConvHull}(w_0, \ldots w_d) \\
        \hat \sigma &= \mathsf{ConvHull}(\phi(u_0), \ldots \phi(u_d)) \\
        \tau_{x, R} &= \Diff_x \phi( \tau^R)
    \end{align*}
    The simplex $\hat \sigma$ is a simplex of the rectified triangulations $\psi_\ell$ of Proposition \ref{prop:B}. We would like to bound its diameter and thickness. This will be done via controlling the quality of $\bar \sigma$, and then by quantifying the approximation $\hat \sigma \approx \bar \sigma$. Since $\bar \sigma = \phi(x) + \l^{-1} \cdot \Diff_x \phi(\tau^R) = \phi(x) + \l^{-1} \tau_{x, R}$, the simplices $(\bar \sigma, \tau_{x, R})$ are affine-equivalent. Therefore the following holds:
    \begin{align} \label{eqn:proof_thmb12_2}
        \diam(\bar \sigma)& = \l^{-1} \cdot \diam(\tau_{x, R}) \nonumber \\
        \thick(\bar \sigma) &=\thick(\tau_{x, R})
    \end{align}
    By the triangle inequality and Lemma \ref{lem:calculus1}, we get the following bound on differences in edge lengths of simplices $\hat \sigma$ and $\bar \sigma$ (here $L_{ij}(\sigma)$ is the edge length of a simplex $\sigma$ across vertices $(i, j)$):
    \begin{align} \label{eqn:proof_thmb12}
         \big| L_{ij}(\hat \sigma) - L_{ij}(\bar \sigma) \big| \nonumber 
        &\le  \big \| (\phi(u_i) - \phi(u_j)) - (w_i - w_j) \big \| \nonumber \\
        &\le \begin{aligned} &\| \phi(x + \l^{-1} v_i^R) - \phi(x) - \Diff_{x} \phi(\l^{-1} v_i^R )) \| +\\[-1ex]&\qquad \| \phi(x + \l^{-1} u_j) - \phi(x) - \Diff_{x} \phi(\l^{-1} v_j^R ) \|\end{aligned} \nonumber \\
        &\le  h_\phi \cdot \diam(\tau)^2 \cdot \l^{-2} \nonumber \\ 
        &\le  h_\phi \cdot d^2 \cdot \diam(\sigma_0)^2 \cdot \l^{-2}
    \end{align}
    Define the following constants:
    \begin{align} \label{eqn:scr_dt}
        & \mathscr{D}_\phi^- = \inf_{\tau,x,R} \diam(\tau_{x, R}), \quad \mathscr{D}_\phi^+ = \sup_{\tau,x,R} \diam(\tau_{x, R}) \nonumber \\
        & \mathscr{T}_\phi^- = \inf_{\tau,x,R} \thick(\tau_{x, R}), \quad \mathscr{T}_\phi^+ = \sup_{\tau,x,R} \thick(\tau_{x, R}) 
    \end{align}
    where the infimum and supremum are taken over $(\tau,x,R) \in Q(\sigma_0) \times K \times O(d)$. This is a compact set, implying that $\mathscr{D}_\phi^\pm > 0$ and $\mathscr{T}_\phi^\pm > 0$. By Equations \eqref{eqn:proof_thmb12_2} and \eqref{eqn:proof_thmb12}, we get the following bound on the diameter:
    \begin{align*}
        \diam(\hat \sigma) \le \diam(\tau_{x, R}) \cdot \frac1{\ell} + h_\phi \cdot d^2 \cdot \diam(\sigma_0)^2 \cdot \frac1{\ell^2}
    \end{align*}    
    Therefore,
    \begin{align*}
        \ell \ge \frac{h_\phi d^2 \diam(\sigma_0)^2}{\mathscr{D}_\phi^+} \implies \diam(\hat \sigma) \le \frac{2 \mathscr{D}_\phi^+}{\ell}
    \end{align*}
    We also apply the quality control lemma of \cite{riemannian_simplices} (Lemma 25) to bound thickness:
    \begin{align*}
        \thick(\hat \sigma) \ge & \frac{4 \thick(\bar \sigma) }{5\sqrt{d}} \bigg( 1 - \frac{4 h_\phi d^2 \diam(\sigma_0)^2 \ell^{-2} }{\diam (\bar \sigma) \thick(\bar \sigma)^2 } \bigg) \\
        =& \frac{4 \thick(\tau_{x,R}) }{5\sqrt{d}} \bigg( 1 - \frac{4 h_\phi d^2 \diam(\sigma_0)^2 \ell^{-1} }{\diam (\tau_{x,R}) \thick(\tau_{x,R})^2 } \bigg)
    \end{align*}
    Then the following implication holds:
    \begin{align*}
        \ell \ge \frac{8h_\phi d^2 \diam(\sigma_0)^2 }{\mathscr{D}^-_\phi (\mathscr{T}^-_\phi)^2 } \implies \thick(\hat \sigma) \ge \frac{2 \mathscr{T}_\phi^- }{5 \sqrt{d}}
    \end{align*}
    This proves Proposition \ref{prop:B}. The constants $c_1, c_2, \l_0$ in the theorem are given as follows:
\begin{align} \label{eqn:constc012}
    c_1 = 2 \mathscr{D}_\phi^+, \quad c_2 = \frac{2 \mathscr{T}_\phi^- }{5 \sqrt{d}}, \quad \l_0 = \max \bigg( \gamma, \> \frac{8h_\phi d^2}{\mathscr{D}^-_\phi \mathscr{T}^-_\phi}\cdot \max_{\sigma_0 \in K} \diam(\sigma_0)^2 \bigg)
\end{align}
Here $\gamma$ is the constant appearing in Proposition \ref{prop:B}.

\bigskip

\section{Proofs of technical lemmas}
\label{section: proofs_technical}

We present proofs of technical results that were postponed from earlier.

\bigskip
\begin{appendixproof}{Lemma \ref{lem:rips_cech_interleaving}}
    The following formulas hold (the \v Cech radius $r_{i, \Cech}(\pty)$ is a true minimum\footnote{Set $A_r = \cap_{y \in \pty} \ball(y, r) \neq \emptyset$. Then $\cap_{0 \le r \le 1} A_r \subseteq A_1$ is non-empty, since $A_1$ is compact. This tells us that the infimum is in fact a minimum.}.):
    \begin{align*}
        r_{i, \rips}(\pty) =& \max \big\{ \dist_i(y_1,y_2) \>|\> y_1 \neq y_2 \in \pty \big\} \\
        r_{i, \Cech}(\pty) =& \min \big\{ r \>|\> \cap_{y \in \pty} \ball(y, r) \neq \emptyset \big\}
    \end{align*}
    For the Rips case, the claim \eqref{eqn:r12_filt} follows straightforwardly from the above formula. For the \v Cech case, consider the following implications:
    \begin{align*}
        & R \ge r_{1,\Cech}(\pty) \implies \bigcap_{y \in \pty} \ball_{\dist_1}(y, R) \neq \emptyset \implies \bigcap_{y \in \pty} \ball_{\dist_2} (y, (1+\delta) R ) \neq \emptyset \implies (1+\delta) R \ge r_{2,\Cech}(\pty)
    \end{align*}
    By considering all possible values of $R$, this implies $r_{2,\Cech}(\pty) \le (1+\delta) r_{1,\Cech}(\pty)$. By switching the roles of $(\dist_1, \dist_2)$ and applying the above chain of implications, we also obtain $r_{1,\Cech}(\pty) \le (1+\delta) r_{2,\Cech}(\pty)$. This gives the claim for the \v Cech case.
\end{appendixproof}

\bigskip
\begin{appendixproof}{Lemma \ref{lem:interleaving}}
    Define the logarithmically rescaled filtration $L_r^{(i)} = K_{e^r}^{(i)}$. By Lemma \ref{lem:rips_cech_interleaving}, $\dist_2 / \dist_1 \in [(1+\Delta)^{-1}, (1+\Delta)]$ implies that $(L_\bullet^{(1)}, L_\bullet^{(2)})$ are $\log (1+\Delta)$-interleaved. By the isometry theorem of persistence modules \cite{bauer_lesnick}, there is a bound on the bottleneck distance:
    \begin{align*}
        \dist_{\mathsf{B}}(\pd_k(L_\bullet^{(1)}), \> \pd_k(L_\bullet^{(2)})) \le \log (1+\Delta)
    \end{align*}
    We recall that the bottleneck distance is realized by a matching between the points in the respective diagrams. 
    Therefore there exist decompositions $\pd_k(L_\bullet^{(i)}) = \wtilde{S}_i \sqcup \wtilde{S}_i'$ for $i=1,2$ such that the following properties hold. Firstly there are enumerations $\wtilde{S}_i = \{(\log b_j^{(i)}, \log d_j^{(i)}) \>|\> 1 \le j \le m\}$ for $i=1,2$ so that the following inequalities hold:
    \begin{align*}
        & \big| \log d_j^{(2)} - \log d_j^{(1)} \big| \le \log (1+\Delta) \iff d_j^{(2)} / d_j^{(1)} \in [(1+\Delta)^{-1}, (1+\Delta)] \\
        & \big| \log b_j^{(2)} - \log b_j^{(1)} \big| \le \log (1+\Delta) \iff b_j^{(2)} / b_j^{(1)} \in [(1+\Delta)^{-1}, (1+\Delta)]
    \end{align*}
    which implies:
    \begin{align*}
        \frac{d_j^{(2)}/b_j^{(2)}}{d_j^{(1)} / b_j^{(1)}} \in [(1+\Delta)^{-2}, (1+\Delta)^2]
    \end{align*}
    For the rest, every $(\log b, \log d) \in \wtilde{S}_i'$ satisfies $|\log d - \log b | \le 2\log (1+\Delta)$, i.e. $d/b \le (1+\Delta)^2$. 
    
    Without the logarithmic scaling, there are equivalent decompositions $\pd_k(K_\bullet^{(i)}) = S_i \sqcup S_i'$ and matched enumerations $S_i = \{(b_j^{(i)}, d_j^{(i)}) \>|\> 1 \le j \le m \}$ for $i=1,2$ so that the above bounds carry over. Denote $\Pi_k^{(i)} = \Pi_k(K_\bullet^{(i)})$ for brevity. By the matching, the following holds whenever $\alpha > (1+\Delta)^2$:
    \begin{align*}
        & \Pi_k^{(2)} \cap [(1+\Delta)^{2} \alpha, \infty) \>\>\subseteq\>\> \Pi_k^{(1)} \cap [\alpha, \infty) \>\>\subseteq\>\> \Pi_k^{(2)} \cap [(1+\Delta)^{-2} \alpha, \infty) \\
        \implies & \bigg| \Pi_k^{(1)} ([\alpha, \infty)) - \Pi_k^{(2)} ([\alpha, \infty)) \bigg| \le \Pi_k^{(2)} ([(1+\Delta)^{-2} \alpha, (1+\Delta)^2 \alpha))
    \end{align*}
    By symmetry, the minimum over $i=1,2$ on the right hand side of the main claim is also obtained. 
\end{appendixproof}

\begin{appendixproof}{Lemma \ref{lem:calculus1}}
    Compute the following:
    \begin{align*}
        \frac{\diff}{\diff t}\bigg|_{t=t_0} \Diff_{x + tu} \phi_i(v) = \sum_{j,k=1}^d u_j v_k \frac{\partial^2 \phi_i}{\partial t_j \partial t_k}\bigg|_{t=t_0} = u^\top \hess_{x(t_0)}(\phi_i) v
    \end{align*}
    where $x(t) = x+tu$. Also,
    \begin{align*}
        \bigg\| \frac{\diff}{\diff t}\bigg|_{t=t_0} \Diff_{x + tu} \phi(v) \bigg\|^2 = \sum_{i=1}^D \bigg( \frac{\diff}{\diff t}\bigg|_{t=t_0} \Diff_{x + tu} \phi_i(v) \bigg)^2 = \sum_{i=1}^D \bigg( u^\top \hess_{x+t_0 u}(\phi_i) v \bigg)^2 \le h_\phi^2 \|u\|^2 \|v\|^2
    \end{align*}
    Let $F_v(t) = \Diff_{x + t(y-x)} \phi(v)$. By the above, $\|F'_v(t_0)\| \le h_\phi \| y-x \| \|v\|$. 
    
    Therefore, both $A = \| \phi(y) - \phi(x) - \Diff_x \phi(y-x) \|$ and $B = \| \Diff_y \phi(v) - \Diff_x \phi(v) \|$ satisfy the following:
    \begin{align*}
        A =& \bigg\| \int_0^1 (F_{y-x}(t) - F_{y-x}(0)) \diff t \bigg\| = \bigg\| \int_0^1 \int_0^t F'_{y-x}(s) \diff s \diff t \bigg\| \le \frac {h_\phi} 2 \|y-x\|^2 \\
        B = & \| F_{v}(1) - F_{v}(0) \| = \bigg\| \int_0^1 F'_{v}(t) \diff t \bigg\| \le h_\phi \|y-x\| \|v\|
    \end{align*}    
\end{appendixproof}

\bigskip
\begin{appendixproof}{Lemma \ref{lem:inward_cone_decomp}}
\label{proof: inward_cone_decomp}
    If $\sigma$ is an affine $d$-simplex, we may consider $\phi' = \phi \circ \psi$ where $\psi: \regsimp{d} \rightarrow \sigma$ is a linear isomorphism from the regular $d$-simplex, so that $\phi': \regsimp{d} \rightarrow \sigma$ is a $C^2$ embedding. Therefore without loss of generality, we can assume that $\sigma = \regsimp{d}$. The main idea of the proof is to pullback the objects to the standard simplex, prove the decomposition there  \eqref{eqn:cone_decomposition_regular} and push them forward by $\phi$.

    Let $\sigma = \regsimp{d}$, let $x \in \partial \tilde \sigma$, and let $\vecbf t \in \partial \sigma$ be the point such that $\phi(\vecbf t) = x$. Without loss of generality, we may set $\vecbf t = (t_0, \ldots t_d, 0, \ldots 0)$ where $t_i > 0$ and $\sum_i t_i = 1$. Define the following pullbacks: $\mathcal{C} = \phi^*(\incone_x \tilde \sigma)$, $V_{\tanspc} = \phi^*(\tanspc_x \tilde \sigma)$ and $V_{\normspc} = \phi^*(\normspc_x \tilde \sigma)$ (the pullback is defined by applying the inverse of the differential of $\phi$ at $x$). Since $\phi$ is a $C^2$ embedding, the pullback preserves dimension and transversality\footnote{In this context, linear subspaces $(V, W)$ intersect transversally if $\dim(V\cap W) = 0$.} of the tangent and normal spaces. Thus $(V_{\tanspc}, V_{\normspc})$ intersect transversally, although not necessarily orthogonally. We then have $V_{\tanspc} \oplus V_{\normspc} = V$, where $V \cong \RR^D$ is the set of all vectors $(s_0, \ldots s_D) \in \RR^{D+1}$ such that $\sum_i s_i = 0$. Since $\vecbf t = (t_0, \ldots t_d, 0, \ldots 0)$ and since the pullback by $\phi$ preserves tangentiality, the following holds:
    \[ V_{\tanspc} = \bigg\{ (s_0, \ldots s_d, 0, \ldots 0) \>\big|\> \sum_i s_i = 0 \bigg\} \]
    Recall that we defined $\mathcal{C} = \phi^*( \incone_x \tilde\sigma)$ above. Since the inward cone is the set of all initial velocity vectors of curves staying in $\tilde \sigma$, we have $\mathcal{C} = \phi^*( \incone_x \tilde\sigma) = \incone_{\vecbf t} \sigma$. By elementary calculation for the standard simplex $\sigma = \regsimp{d}$, this is given as follows:
    \begin{align*}
        \mathcal{C} = \incone_{\vecbf t} \sigma = \bigg\{ (s_0, \ldots s_D) \>\big|\> \sum_i s_i = 0 \text{, and } s_{d+1} \ge 0, \ldots s_D \ge 0 \bigg\}
    \end{align*}
    Due to the linear decomposition $V_{\normspc} \oplus V_{\tanspc} = V$, we have the following homeomorphism:
    \begin{align} \label{eqn:cone_decomposition_regular}
        (\mathcal{C} \cap V_{\normspc}) \times V_{\tanspc} \cong & \>\> \mathcal{C} \nonumber \\
        (v, w) \mapsto & \>\> v+w
    \end{align}
    This decomposition can be seen from the following general argument. Let $\{f_1, \ldots f_{D-d} \}$ be linear forms defined by $f_i(s_0, \ldots s_D) = s_{d+i}$. Then the cone $\mathcal{C}$ is the set of all $\vecbf s \in V$ for which $f_1(\vecbf s) \ge 0, \ldots f_{D-d}(\vecbf s) \ge 0$. Also, for every $\vecbf s \in V_{\tanspc}$, we have $f_1(\vecbf s) = \cdots = f_{D-d}(\vecbf s) = 0$. Therefore for every $(v, w) \in (\mathcal{C} \cap V_{\normspc}) \times V_{\tanspc}$, we have $f_i(v+w) = f_i(v) + f_i(w) = f_i(v) \ge 0$, so that $v+w \in \mathcal{C}$. Conversely for every $u \in \mathcal{C}$, if $u = v+w$ for $(v,w) \in V_{\normspc} \times V_{\tanspc}$, we have $f_i(v) = f_i(u-w) = f_i(u) \ge 0$, so that $v \in V_{\normspc} \cap \mathcal{C}$. This shows the above homeomorphism.

    Finally we show the main claim. Note that $\phi^*(\nincone_x \tilde \sigma) = \phi^*(\incone_x \tilde \sigma \cap \normspc_x \tilde \sigma) = \mathcal{C} \cap V_{\normspc}$. Since  $V_{\tanspc}, V_{\normspc}, \mathcal{C}$ are defined as pullbacks, we obtain the following by applying the pushforward $\phi_*$:
    \[ \phi_*(V_{\tanspc}) = \tanspc_x \tilde \sigma, \>\>\> \phi_*(V_{\normspc}) = \normspc_x \tilde \sigma, \>\>\> \phi_*(\mathcal C) = \incone_x \tilde \sigma, \>\>\> \phi_*(\mathcal C \cap V_{\normspc}) = \nincone_x \tilde \sigma \]
    By the homeomorphism $(\mathcal C \cap V_{\normspc}) \times V_{\tanspc} \cong \mathcal{C}$, we then have the following, as claimed:
    \begin{align*}
        \nincone_x \tilde \sigma \times \tanspc_x \tilde \sigma = \phi_* (\mathcal{C} \cap V_{\normspc}) \times \phi_* (V_{\tanspc}) \cong \phi_*(\mathcal{C}) = \incone_x \tilde \sigma
    \end{align*}
\end{appendixproof}

\begin{appendixproof}{Lemma \ref{lem:variational1}}
\label{proof: variational1}
    Before working with $g(\lambda)$, first note that the conditions $\epsilon \le \sin(\theta/4)$ and $\theta \le \pi$ imply $\epsilon \le \sin(\pi/4)$. The derivative of $g(\lambda)$ is given by:
    \[ g'(\lambda) = \frac{\lambda - \ell \cos \theta}{\sqrt{\ell^2 + \lambda^2 - 2\ell \lambda \cos \theta}} - \epsilon \]
    The equation $g'(\lambda) = 0$ has one root over the range $-\infty < \lambda < \infty$. This can be seen by solving for an auxiliary  variable $\tilde \lambda = \lambda - \ell \cos \theta$ and tracking sign correctly when taking a square or squareroot. This root $\lambda_0$ is given as follows:
    \begin{align*}
        g'(\lambda_0) = 0 \iff \lambda_0 = \ell \cos \theta + \frac{\epsilon \ell \sin \theta}{\sqrt{1-\epsilon^2}}
    \end{align*}
    We claim that $\lambda_0 \ge 0 \iff \cos \theta \ge -\epsilon$. To see this, first assume $\theta \in (0, \pi)$, so that $\sin \theta > 0$. We have $\lambda_0 \ge 0$ iff $\cos \theta + \epsilon \sin \theta / \sqrt{1-\epsilon^2} \ge 0$. Dividing both sides by $\sin \theta$, we get $f(\epsilon) \le f(- \cos \theta)$ where $f(t) = t/\sqrt{1-t^2}$ is an increasing function over $t \in (0,1)$. Therefore $\lambda_0 \ge 0 \iff \cos\theta \ge -\epsilon$ when $\theta \in (0, \pi)$. The case $\theta = \pi$ can be checked individually.

   \noindent \underline{Case 1.} Suppose $\cos \theta \ge -\epsilon$. We claim that $\lambda_0 \in [0, \epsilon \reach]$. By the previous calculation we have $\lambda_0 \ge 0$. Setting $t = \arcsin(\epsilon)$, the following holds:
    \[ \lambda_0 = \ell \cos \theta +  \frac{\ell \epsilon \sin \theta}{ \sqrt{1-\epsilon^2}} = \frac{\ell \cos(\theta - t)}{\cos t} \le \frac{\ell \cos t}{\cos t} \le \ell \le \epsilon \reach \]
    In the above, the second equality is by a trigonometric addition formula, and the first inequality is due to\footnote{Since arcsin is an increasing function, the condition $\epsilon \le \sin(\theta/4)$ implies that $t = \arcsin(\epsilon) \le \arcsin(\sin(\theta/4)) = \theta/4$. Here $\arcsin(\sin(\theta/4)) = \theta/4$ because $\theta/4 \in [0, \pi/4]$. Thus $t \le \theta/4 \le \theta/2 \le \theta$, so that $\cos(\theta - t) \le \cos(t)$.} the condition $\epsilon \le \sin(\theta/4)$. The critical value at $\lambda_0$ is given as follows, along with a lower bound:
    \begin{align*}
        g(\lambda_0) = \sqrt{1-\epsilon^2} \ell \sin \theta - \epsilon \ell (1+\cos \theta) \ge (1-\epsilon) \ell \sin \theta - 2\epsilon \ell \ge \ell \sin \theta - 3 \epsilon \ell
    \end{align*}
    By checking the extremal values\footnote{In the range $\lambda \in [0, \infty)$, we have $g(0) = (1-\epsilon) \ell$, which satisfies $g(0) \ge g(\lambda_0)$ because $g(\lambda_0) = \sqrt{1-\epsilon^2} \ell \sin \theta - \epsilon \ell (1+\cos\theta) \le \ell - \epsilon \ell - \epsilon \ell \cos \theta < (1-\epsilon) \ell = g(0)$. Also we have $g(\lambda) = \sqrt{\ell^2 \sin^2 \theta + \tilde \lambda^2} - \epsilon \tilde \lambda - \epsilon \ell (1+\cos \theta) \ge (1-\epsilon )\tilde \lambda - \epsilon \ell (1+\cos \theta)$ where $\tilde \lambda = \lambda - \ell \cos \theta$. Thus $g(\lambda)$ grows unboundedly large as $\lambda \rightarrow \infty$.}, we can see that the above critical value is the minimum over the range $\lambda \in [0, \epsilon \reach]$. The assumption $\epsilon \le \sin(\theta/4)$ implies\footnote{The conditions $\epsilon \le \sin(\theta/4)$ and $\theta \le \pi$ imply that $\epsilon \le \sin(\pi/4)$. Then the assumption $\cos \theta \ge -\epsilon$ for Case 1 is equivalent to $\theta \le (\pi/2) + \arcsin(\epsilon)$, which implies $\theta \le (\pi/2) + (\pi/4) = (3/4)\pi$. Thus $\theta_+ = \theta$.} that $\theta_+ = \theta$, and thus the claimed lower bound follows for Case 1.

     \noindent \underline{Case 2.} Suppose $\cos \theta < -\epsilon$. Then $\lambda_0 < 0$. Since $g'(0) = -\cos\theta -\epsilon > 0$, $g$ is an increasing function, and the global minimum is reached at $0$. This value is simply $g(0) = (1-\epsilon) \ell$. Then the claimed lower bound $g(0) \ge \ell (\sin \theta_+ - 3\epsilon)$ is easily checked to be true.

    We thus covered both Case 1 and Case 2, and the main claim is proven. We additionally summarize the argmin and the minimum for completeness.
    \begin{align*}
        \operatorname{argmin}_{0 \le \lambda \le \epsilon \reach} g(\lambda) =& \begin{cases}
            \ell \cos \theta + \frac{\epsilon \ell \sin \theta}{\sqrt{1-\epsilon^2}} & \text{ If $\cos \theta \ge -\epsilon$} \\
            0 & \text{If $\cos \theta \le -\epsilon$}
        \end{cases} \\
        \min_{0 \le \lambda \le \epsilon \reach} g(\lambda) =& \begin{cases}
            \sqrt{1-\epsilon^2} \ell \sin \theta - \epsilon \ell (1+\cos \theta) & \text{ If $\cos \theta \ge - \epsilon$} \\
            (1-\epsilon)\ell & \text{If $\cos \theta < -\epsilon$}
        \end{cases}
    \end{align*}
\end{appendixproof}

\begin{appendixproof}{Lemma \ref{lemma: interference1}}
\label{proof: interference1}
    Firstly since $\ell \le (\sqrt{2}-1) \reach$ and $\sqrt{\lambda_\normspc^2 + \lambda_\tanspc^2} \le (\sqrt{2}-1)\reach$, we can apply the geodesic control (Lemma \ref{lem:geodesic_control}) to obtain $\|x - \bar x \| \le \reach^{-1} \ell^2$ and $\|y - \bar y \| \le \reach^{-1} (\lambda_\normspc^2 + \lambda_\tanspc^2)$. Applying this and rewriting some expressions, we have the following:
    \begin{align*}
        \|x-y\| \ge & \|\bar x - \bar y \| - \reach^{-1}( \ell^2 + \lambda_\normspc^2 + \lambda_\tanspc^2) \\
        =& \|\ell u - \lambda_\normspc v_\normspc - \lambda_\tanspc v_\tanspc \| - \reach^{-1}( \ell^2 + \lambda_\normspc^2 + \lambda_\tanspc^2) \\
        =& \sqrt{w^2 + \lambda_\tanspc^2}  - \reach^{-1} \lambda_\tanspc^2 - \reach^{-1} (\ell^2 + \lambda_\normspc^2), \quad \text{where } w \eqdef \|\ell u - \lambda_\normspc v_\normspc \| \\
        =& \reach \cdot h \big( \reach^{-2} (w^2 + \lambda_\tanspc^2) \big) + \reach^{-1} (w^2 - \ell^2 - \lambda_\normspc^2), \quad \text{where } h (t) \eqdef \sqrt{t} - t
    \end{align*}
    We claim that the above lower bound is minimized at $\lambda_\tanspc = 0$ (while fixing $\ell$ and $\lambda_\normspc$ constant). To see this, observe that $h(t) = \sqrt{t} - t$ is unimodal over $t \in [0, 1]$ with its maximum reached at $t = 1/4$. Meanwhile we have $\reach^{-2}(w^2 + \lambda_\tanspc^2) \le (2/5)^2 + (1/5)^2 \le 1/4$ by the assumptions on their range. Therefore $h \big( \reach^{-2} (w^2 + \lambda_\tanspc^2) \big)$ is an increasing function on our fixed range of $\lambda_\tanspc$ and therefore the above lower bound reaches its minimum at $\lambda_\tanspc = 0$.
    
    Setting $\lambda_\tanspc = 0$, the above lower bound becomes the following (recall that $\theta = \angle(u, v_\normspc)$):
    \begin{align*}
        \|x-y \| \ge& w - \reach^{-1} (\ell^2 + \lambda_\normspc^2) \\
        =& \sqrt{\ell^2 + \lambda_\normspc^2 - 2 \ell \lambda_\normspc \cos \theta} - \reach^{-1} (\ell^2 + \lambda_\normspc^2) \\
        \ge & \sqrt{\ell^2 + \lambda_\normspc^2 - 2 \ell \lambda_\normspc \cos \theta} - \epsilon (\ell + \lambda_\normspc)
    \end{align*}
    where we used the assumptions $\ell, \lambda_\normspc \in [0, \epsilon \reach]$. This now fits the form in Lemma \ref{lem:variational1}, so that the above lower bound is further lower bounded as:
    \begin{align*}
        \|x-y\| \ge \ell ( \sin \theta_+ - 3 \epsilon) \text{, where } \theta_+ = \min\bigg( \theta, \frac{3\pi} 4 \bigg)
    \end{align*}
\end{appendixproof}

\pagebreak

\bibliographystyle{amsplain}
\bibliography{refs}

\pagebreak
\appendix

\section{Notations} 
\label{appendix: notations}

We summarize the notations and conventions used in this article.

\begin{figure}[h]
\centering
    \begin{tabular}{|l|l|}
    
        \hline
        \textbf{Symbol} & \textbf{Description}

        \\ \hline \hline
        $\ptx, \pty$ & Set of points

        \\ \hline
        $\rips, \Cech$ & Rips and \v Cech filtrations

        \\ \hline
        $\|\vecbf v\|_p$ & The $p$-norm of a vector $\vecbf v$, where $\|\vecbf v\| = \|\vecbf v\|_2$.
        
        \\ \hline
        $\vecbf e_i$ & The standard basis vectors $(0, \ldots 0, 1, 0, \ldots 0)$
        
        \\ \hline
        $\mathbf{1}[\mathcal{A}]$ & The indicator function for a statement $\mathcal{A}$.

        \\ \hline
        $\ball_M(\vecbf x, r)$ & The closed ball of radius $r$, centered at $\vecbf x$, in the metric space $M$.

        \\ \hline
        $\dist_\euclidean(x, y)$ & The Euclidean distance between points $x, y$.
        
        \\ \hline
        $\dist_M(x, y)$ & The distance between points $x, y$ in a metric space $M$.

        \\ \hline
        $d$ & The (intrinsic) dimension of a triangulable space

        \\ \hline
        $D$ & The (ambient) dimension of $\RR^D$ that contains a triangulable space

        \\ \hline
        $\Diff_x f$ & The differential of a function $f$ at $\vecbf x$.

        \\ \hline
        $\reach_S$ & The reach of a set $S$.

        \\ \hline
        $\pointprocess_{nf}$ & Poisson point process with intensity function $nf$

        \\ \hline
        $\func(\ptx)$ & Functional evaluated at $\ptx$

        \\ \hline
        $\overline{\func}(f)$ & Expected functional, defined as $\EE[ \func(\pointprocess_f)]$

        \\ \hline
    \end{tabular}
\end{figure}

\section{Freudenthal-Kuhn subdivision} \label{appendix: fkt_theory}

Let us denote $\indexset{a, b} = \big\{ n \in \ZZ \>\big|\> a \le n \le b \big\}$. We define the $I$-face $\partial_I \ordsimp{d}$ as follows. Let $d \in \ZZ_+$ be a positive integer and let $I = \{u(0), \ldots u(m)\} \subseteq \{0, \ldots d \}$ where we assume $u(0) < \cdots < u(m)$. Then we define:
\begin{align*}
    \partial_I \ordsimp{d} \eqdef & \bigg\{ (\underbrace{1, \ldots 1}_{u(0)}, \underbrace{t_1, \ldots t_1}_{u(1)-u(0)}, \ldots \underbrace{t_m, \ldots t_m}_{u(m)-u(m-1)}, \underbrace{0, \ldots 0}_{d-u(m)}) \>\bigg|\> 1 \ge t_1 \ge \cdots \ge t_m \ge 0 \bigg\}
\end{align*}

The FK triangulation is defined in the following order:
\begin{enumerate}
    \item FK-triangulation of a Euclidean space (Definition \ref{defn:fkt_euclid})
    \item FK-triangulation of a simplex (Definition \ref{defn:fkt_spx})
    \item FK-triangulation of a simplicial complex (Definition \ref{defn:fkt_spcpx})
\end{enumerate}

Let $\permgrp_d$ denote the permutation group on $d$ elements. We will use the notation $\permelt \in \permgrp_d$ for a permutation. The action of $\permgrp_d$ on $\RR^d$ is defined by setting $\permelt \cdot \vecbf e_i = \vecbf e_{\permelt(i)}$ or $(\permelt \cdot (t_1, \ldots t_d) )_{\permelt(i)} = t_i$, or equivalently:
\begin{align*}
    \permelt \cdot (t_1, \ldots t_d) = (t_{\permelt^{-1}(1)}, \ldots t_{\permelt^{-1}(d)})
\end{align*}
Given $I = \{u(0), \ldots u(m)\} \subseteq \indexset{0, d}$ with $u(0) < \cdots < u(m)$, the \textit{coordinate projection} $\pi_I: \RR^d \rightarrow \RR^m$ is the following map (note that $t_{u(0)}$ does not appear below even though $u(0) \in I$):
\begin{align*}
    \pi_I(t_1, \ldots t_d) = (t_{u(1)}, \ldots t_{u(m)})
\end{align*}

\begin{defn}[FK-triangulation of a Euclidean space] \label{defn:fkt_euclid}
    A \textit{FK-simplex} is a simplex of the form $\vecbf n + \permelt (\partial_I \ordsimp{d})$ for some $\vecbf n \in \ZZ^d$, $\permelt \in \permgrp_d$, and $I \subseteq \indexset{0, d}$. 
    The \textit{FK-triangulation} of $\RR^d$ is a $d$-dimensional infinite simplicial complex whose simplices are the set of all FK-simplices. We denote the FK-triangulation of $\RR^d$ by the notation $\mathsf{FK}(\RR^d)$.
\end{defn}

The defining equations of a FK-simplex are given as follows:

\begin{lem}
    Suppose $\vecbf n \in \ZZ^d$, $\permelt \in \permgrp_d$, and $I = \{u(0), \ldots u(m)\} \subseteq \indexset{0, d}$ with $u(0) < \cdots < u(m)$. Then $(t_1, \ldots t_d) \in \vecbf n + \permelt (\partial_I \simp_d)$ holds iff the following hold:
    \begin{align} \label{eqn:fksimp_defn}
        1 & = t_{\permelt(1)} = \cdots = t_{\permelt(u(0))} \ge t_{\permelt(u(0)+1)} = \cdots = t_{\permelt(u(1))} \ge \cdots \nonumber \\
        & \cdots \ge t_{\permelt(u(m-1)+1)} = \cdots = t_{\permelt(u(m))} \ge t_{\permelt(u(m)+1)} = \cdots = t_{\permelt(d)} = 0
    \end{align}
\end{lem}
\begin{proof}
    The $m$-simplex $\partial_I \ordsimp{d}$ is spanned by the vectors $\vecbf e_{0} + \cdots + \vecbf e_{u(i)}$ for $i=0, \ldots m$. Therefore $(t_1, \ldots t_d) \in \partial_I \ordsimp{d}$ holds iff the following is true:
    \begin{align*}
        1 & = t_1 = \cdots = t_{u(0)} \ge t_{u(0)+1} = \cdots = t_{u(1)} \ge \cdots \\
        & \cdots \ge t_{u(m-1)+1} = \cdots = t_{u(m)} \ge t_{u(m)+1} = \cdots = t_{d} = 0
    \end{align*}
    Since $\vecbf t \in \vecbf n + \permelt (\partial_I \simp_d)$ holds iff $\permelt^{-1}(\vecbf t - \vecbf n) \in \partial_I \simp_d$ holds, the main claim follows.
\end{proof}

A FK-simplex satisfies the following nice properties:

\begin{prop} \label{prop:fksimp_properties}
    The following are true.
    \begin{enumerate}
        \item Vertices of the FK-simplex $\vecbf n + \permelt(\partial_I \ordsimp{d})$ are $(\vecbf u_0, \ldots \vecbf u_m)$ where $\vecbf u_i = \sum_{j=1}^{u(i)} \vecbf e_{\permelt(j)}$ and $I = \{u(0), \ldots u(m) \} \subseteq \indexset{0, d}$ with $u(0)< \cdots < u(m)$.
        \item The intersection of a FK-simplex with the affine span of another FK-simplex is a FK-simplex.
        \item A coordinate projection of a FK-simplex is a FK-simplex. 
    \end{enumerate}
\end{prop}
\begin{proof}
    (1) This follows directly from the definition of $\partial_I \ordsimp{d}$.
    
    (2) It suffices to prove this for the intersection $(\vecbf n + \permelt \sigma) \cap \mathsf{span}(\tau)$, where $\sigma, \tau$ are both faces of the standard order simplex $\ordsimp{d}$. First note that $\sigma = \partial_I \ordsimp{d}$ for some $I$, which means that $\vecbf n + \permelt \sigma$ is defined by \eqref{eqn:fksimp_defn}. The affine span $\mathsf{span}(\tau)$ is defined by expressions of the form $t_i = t_j$ and $t_k = m$ for some tuples $(i, j)$ and $(k, m)$ (here $m \in \{0, 1\}$). Applying these constraints to \eqref{eqn:fksimp_defn} produces another set of equalities and inequalities of the same form as \eqref{eqn:fksimp_defn}, just with different $(\vecbf n, \permelt, I)$. Therefore there exist $\vecbf n', \permelt, \sigma'$ such that $(\vecbf n + \permelt \sigma) \cap \mathsf{span}(\tau) = \vecbf n' + \permelt' \sigma'$ holds, as claimed.

    (3) This follows from considering projections of the vertices. Let $I_j = \permelt \cdot \indexset{u(j-1)+1, u(j)}$. Let $J \subseteq \indexset{1, d}$ be such that $|J| = \ell$ and also let $\eta: \indexset{1, \ell} \rightarrow J$ be the unique increasing bijection. The projection of the $i$-th vertex of $\vecbf n + \permelt \sigma$ is given as follows:
    \begin{align*}
        \pi_J \bigg( \sum_{j \le i} \sum_{k \in I_j } \vecbf e_k \bigg) =& \sum_{j \le i} \sum_{k \in \eta^{-1}(I_j \cap J)} \vecbf e_k
    \end{align*}
    The sets $J_j = \eta^{-1}(I_j \cap J)$ satisfy $J_0 \sqcup \cdots \sqcup J_m = \indexset{1, \ell}$. There exists a permutation $\permelt' \in \permgrp_\ell$ and values $u'(0), \ldots, u'(m)$ such that $(\permelt')^{-1} J_j = \indexset{u'(j-1)+1, u'(j)}$, so that $J_j = \permelt' \cdot \indexset{u'(j-1)+1, u'(j)}$. This is of the same form as the vertices described in (1), which proves the claim that the projected vertices form another FK-simplex.
\end{proof}

\begin{prop} \label{prop:fkt_ksimp}
    The set $k \cdot \ordsimp{d}$ triangulates into FK-simplices as follows:
    \begin{align} \label{eqn:fkt_kfold_simplex}
        k \cdot \ordsimp{d} = \bigcup_{I \subseteq \indexset{d-1}} \bigcup_{\substack{\vecbf n \in \vecbf N_{d, k} [I] \\ \permelt \in \permgrp_d [I]}} (\vecbf n + \permelt \ordsimp{d})
    \end{align}
    where $\vecbf N_{d, k} [I] \subset \ZZ^d$ and $\permgrp_d [I] \subseteq \permgrp_d$ are defined as:
    \begin{align*}
        \vecbf N_{d, k} [I] =& \bigg\{(n_1, \ldots n_d) \in \ZZ^d \>\bigg|\> k-1 \ge n_1 \ge \cdots \ge n_d \ge 0 \text{, and } (i \in I \iff n_i = n_{i+1}) \bigg\} \\
        \permgrp_d [I] =& \bigg\{ \permelt \in \permgrp_d \>\bigg|\> i \in I \implies \permelt^{-1}(i) < \permelt^{-1}(i+1) \bigg\}
    \end{align*}
\end{prop}
\begin{proof}
    ($\supseteq$) Let $I \subseteq [d-1]$ be an index set, let $\vecbf n \in \vecbf N_{d, k}[I]$ and let $\permelt \in \permgrp_d[I]$. To show that $\vecbf n + \permelt (\ordsimp{d}) \subseteq k \cdot \ordsimp{d}$, it suffices to show the following implication:
    \begin{align} \label{eqn:k_ordsimp1}
        1 \ge t_{\permelt(1)} \ge \cdots \ge t_{\permelt(d)} \ge 0 \implies k \ge t_1 + n_1 \ge \cdots \ge t_d + n_d \ge 0
    \end{align}
    If $i \notin I$, then $n_i > n_{i+1}$, and thus $t_i + n_i \ge n_i \ge 1 + n_{i+1} \ge t_{i+1} + n_{i+1}$. If $i \in I$, then $n_i = n_{i+1}$ and also $\permelt^{-1}(i) < \permelt^{-1}(i+1)$. Writing $i = \permelt(\permelt^{-1}(i))$, we see that $t_i = t_{\permelt(\permelt^{-1}(i))} \ge t_{\permelt(\permelt^{-1}(i+1))} = t_{i+1}$, so that $t_i + n_i = t_i + n_{i+1} \ge t_{i+1} + n_{i+1}$.

    ($\subseteq$) Conversely, suppose that $(\tilde t_1, \ldots \tilde t_d) \in k \cdot \ordsimp{d}$, i.e. $k > \tilde t_1 > \cdots > \tilde t_d > 0$ (note that this is the interior set of the left hand side $k \cdot \ordsimp{d}$). Define $n_i = \lfloor \tilde t_i \rfloor$ and $t_i = \tilde t_i - n_i$. Then $0 \le t_i < 1$ holds and also $k-1 \ge n_1 \ge \cdots \ge n_d \ge 0$. Define $\permelt \in \permgrp_d$ to be any permutation such that $t_{\permelt(1)} \ge \cdots \ge t_{\permelt(d)}$. Define the index set $I \subseteq [d-1]$ as the collection of all $i$ such that $n_i = n_{i+1}$, which implies $\vecbf n \in \vecbf N_{d, k}[I]$. 
    
    Whenever $i \in I$, we have $n_i = n_{i+1}$, so that $t_i - t_{i+1} = (t_i + n_i) - (t_{i+1} + n_{i+1}) = \tilde t_i - \tilde t_{i+1} > 0$, i.e. $t_i > t_{i+1}$. Now since $t_{\permelt(1)} \ge \cdots \ge t_{\permelt(d)}$ and by writing $i = \permelt(\permelt^{-1}(i))$, we see that $i \in I$ implies $\permelt^{-1}(i) < \permelt^{-1}(i+1)$. Therefore $\permelt \in \permgrp_d[I]$. Therefore the given element $(\tilde t_1, \ldots \tilde t_d) \in k \cdot \ordsimp{d}$ also satisfies $(\tilde t_1, \ldots \tilde t_d) \in \vecbf n + \permelt(\ordsimp{d})$ with $\vecbf n \in \vecbf N_{d, k}[I]$ and $\permelt \in \permgrp_d[I]$ for some $I \subseteq [d-1]$. This implies that the interior set $k \cdot \ordsimp{d}$ is the subset of the right hand side. By taking closure of both sides and since the right hand side is a closed set already, we obtain the full inclusion of the left hand side $k \cdot \ordsimp{d}$ to the right hand side.
\end{proof}

\begin{prop} \label{prop:fkt_face_restriction}
    Given $J \subset \indexset{0, d}$ and $k \in \ZZ_+$, the FK-simplex $k \cdot \partial_J \ordsimp{d}$ and the affine span $\mathsf{span}(\partial_J \ordsimp{d})$ each triangulates into a subcomplex of the FK-triangulation of $\RR^d$, and the following hold:
    \begin{align*}
        \mathsf{FK}(\RR^d)|_{\mathsf{span}(\partial_J \ordsimp{d})} =& (\pi_J)_* \big( \mathsf{FK}(\RR^{|J|}) \big) \\
        \mathsf{FK}(\RR^d)|_{k \cdot \partial_J \ordsimp{d}} =& (\pi_J)_* \big( \mathsf{FK}(\RR^{|J|})|_{k \cdot \ordsimp{|J|}} \big)
    \end{align*}
    where $(\pi_J)_*$ is the pushforward along the coordinate projection $\pi_J$.
\end{prop}
\begin{proof}
    Denote $W = \mathsf{span}(\partial_J \ordsimp{d})$ and let $J = \{u(0), \ldots u(m)\} \subseteq \indexset{0, d}$ with $u(0)< \cdots < u(m)$. Firstly by (2) of Proposition \ref{prop:fksimp_properties}, the intersection of any FK-simplex with $W$ is a FK-simplex, and therefore $W$ triangulates into a set of FK-simplices, i.e. a subcomplex of the FK-triangulation of $\RR^d$.

    We now consider the pushforward of the FK-triangulation of $W$ through the map $\pi_J$. The map $\pi_J: W \rightarrow \RR^m$ is a linear isomorphism, so it has an inverse $\pi_J: \RR^m \rightarrow W$ given as:
    \begin{align*}
        \pi_J(t_1, \ldots t_d) =& (t_{u(1)}, \ldots t_{u(m)}) \\
        \iota_J(s_1, \ldots s_m) =& (\underbrace{1, \ldots 1}_{u(0)}, \underbrace{s_1, \ldots s_1}_{u(1)-u(0)}, \ldots \underbrace{s_m, \ldots s_m}_{u(m)-u(m-1)}, \underbrace{0, \ldots 0}_{d-u(m)})
    \end{align*}
    By (2), (3) of Proposition \ref{prop:fksimp_properties}, the following holds for some vectors $\vecbf n', \vecbf n''$, index sets $I', I''$, and permutations $\permelt', \permelt''$ :
    \begin{align*}
        \pi_J \big( W \cap (\vecbf n + \permelt (\partial_I \ordsimp{d}) ) \big) = \pi_J (\vecbf n' + \permelt' (\partial_{I'} \ordsimp{d}) ) = \vecbf n'' + \permelt '' (\partial_{I''} \ordsimp{m})
    \end{align*}
    And by (1) of Proposition \ref{prop:fksimp_properties}, the map $\iota_J$ sends a vertex of a FK-simplex to a vertex of another FK-simplex. Thus the following holds for some index set $I'$ and a permutation $\permelt'$:
    \begin{align*}
        \iota_I ( \vecbf n + \permelt (\partial_I \ordsimp{m})) = \vecbf n' + \permelt' (\partial_{I'} \ordsimp{d})
    \end{align*}
    Therefore, the pushforward of the FK-triangulation of $W$ along the map $\pi_J: W \rightarrow \RR^m$ is the FK-triangulation of $\RR^m$ (and conversely the pushforward of the FK-triangulation of $\RR^m$ along the map $\iota_J: \RR^m \rightarrow W$ is the FK-triangulation of $W$).

    The second statement regarding the restriction to $k \cdot \partial_I \ordsimp{d}$ follows from noting that $\pi_J(\partial_J \ordsimp{d}) = \ordsimp{|J|}$, and also Proposition \ref{prop:fkt_ksimp}.
\end{proof}

Due to the above proposition, we can now define the FK-triangulation of a geometric simplex, although one needs to fix an ordering of the vertices to define one. 

\begin{defn}[FK-triangulation of a simplex] \label{defn:fkt_spx}
    Let $\sigma \subset \RR^D$ be an affine $d$-simplex with an ordering of vertices given by $v^\sigma_\bullet = (v_0^\sigma, \ldots v_d^\sigma)$. The $k$-th FK-triangulation of $(\sigma, v^\sigma_\bullet)$ is defined as follows:
    \begin{align*}
        \FK_k(\sigma, v^\sigma_\bullet) \eqdef \phi_* \big( \mathsf{FK}(\RR^d)|_{k \cdot \ordsimp{d}} \big)
    \end{align*}
    where the FK-triangulation on $k \cdot \ordsimp{d}$ is described by \eqref{eqn:fkt_kfold_simplex}, $\phi$ is the affine linear map defined by setting $\phi(k \cdot (\vecbf e_1 + \cdots + \vecbf e_i)) = v_i^\sigma$ for every $i = 0, 1, \ldots d$, and $\phi_*$ denotes the pushforward. This triangulation is independent of the linear embedding $\sigma \hookrightarrow \RR^D$.
\end{defn}

The following says that the restriction of a FK-triangulation of a simplex is also a FK-triangulation, and it inherits the vertex order. It is the key result of this subsection.

\begin{prop} \label{prop:fkt_welldefined}
    Let $\sigma \subset \RR^D$ be an affine simplex and let $\tau \subset \sigma$ be its face. Let $v^\sigma_\bullet = (v_0, \ldots v_d)$ and $v^\tau_\bullet = (v_{u(0)}, \ldots v_{u(m)})$ be the (unique) orderings of the vertices of $\sigma$ and $\tau$ such that $u(0) < \ldots < u(m)$. Then the following holds:
    \begin{align*}
        \FK_k(\sigma, v_\bullet^\sigma)|_\tau = \FK_k(\tau, v_\bullet^\tau)
    \end{align*}
\end{prop}
\begin{proof}
    Let $\phi: \sigma \rightarrow \ordsimp{d}$ to be the affine linear isomorphism sending $v_i \in \sigma$ to $\vecbf e_1 + \cdots + \vecbf e_i \in \ordsimp{d}$. If we let $J = \{u(0), \ldots u(m) \}$, then $\phi$ sends $(\sigma, \tau)$ to $(\ordsimp{d}, \partial_J \ordsimp{d})$. Thus without loss of generality, we may simply set $(\sigma, \tau) = (\ordsimp{d}, \partial_J \ordsimp{d})$.

    Let $\pi_J: \mathsf{span}(\partial_J \ordsimp{d}) \rightarrow \RR^m$ be the coordinate projection with respect to $J$. By Proposition \ref{prop:fkt_face_restriction}, the FK-triangulation of $\mathsf{span}(\partial_J \ordsimp{d}) \subset \RR^d$ is mapped to the FK-triangulation of $\RR^m$. Therefore it remains to prove that $\pi_J(k \cdot \partial_J \ordsimp{d}) = k \cdot \ordsimp{m}$. By linearity, it suffices to prove the $k=1$ case $\pi_J(\partial_J \ordsimp{d}) = \ordsimp{m}$. But this is trivially true, with the vertex orders correctly mapped between the two simplices. This proves the claim.
\end{proof}

Proposition \ref{prop:fkt_welldefined} implies that the FK-triangulation of a simplicial complex is well-defined after fixing a global vertex order.

\begin{defn}[FK-triangulation of a simplicial complex] \label{defn:fkt_spcpx}
    Let $K$ be a finite geometric simplicial complex, and let $v_\bullet = (v_1, \ldots v_n)$ be an ordering of its vertices. Then the $k$-th FK-triangulation $\FK_k(K, v_\bullet)$ is defined by gluing the triangulations $\FK_k(\sigma, v^\sigma_\bullet)$ for every simplex $\sigma \in K$. Here, $v^\sigma_\bullet = (v_{i(0)}, \ldots v_{i(d)})$ is the unique ordering of the vertices of $\sigma$ that satisfies $i(0)< \ldots <i(d)$.
\end{defn}

Lastly we note the following.

\begin{lem} \label{lem:fkt_isometry_bound}
    For every $\l > 0$, the number of simplices in the simplicial complex $\FK_\l (K, v_\bullet)$ up to Euclidean isometry is at most $\sum_{\sigma \in K} d_\sigma!$, where $d_\sigma = \dim(\sigma)$ and the sum is taken over maximal simplices $\sigma \in K$.
\end{lem}
\begin{proof}
    It suffices to prove the claim for each maximal simplex. Firstly note that the FK-triangulation of the Euclidean space $\RR^d$ contains $d!$ different simplices up to reflection, rotation, and translation. Therefore an affine transformation of its subset contains at most $d!$ Euclidean isometry types of simplices. 
\end{proof}

\section{Necessity of vertex ordering} \label{appendix: fkt_vertex_ordering}

We consider the issue of orientation compatibility arising in FK triangulation. While this does not happen in 2-dimensional FK triangulation, it is an issue from dimension 3 onwards. Consider the 3-dimensional order simplex and its 2nd FK triangulation. The four vertices of the order 3-simplex are given by coordinates
\begin{align*}
    A^o_1 = (0,0,0),\> A^o_2 = (0,0,1), \> A^o_3 = (0,1,1), \> A^o_4 = (1,1,1)
\end{align*}
and each $B_{ij}$ is the midpoint of $(A_i, A_j)$, given explicitly as:
\begin{align*}
    B_{12} =& \bigg(0,0,\frac12\bigg), \> B_{13} = \bigg(0, \frac12, \frac12\bigg), \> B_{14} = \bigg(\frac12, \frac12, \frac12\bigg) \\
    B_{34} =& \bigg(\frac12, 1, 1\bigg), \> B_{24} = \bigg(\frac12, \frac12, 1\bigg), \> B_{23} = \bigg(0, \frac12, 1\bigg)
\end{align*}

\begin{lem} \label{lem:cox_bad1}
    The 2nd FK triangulation of the ordered 3-simplex $(A_1^o A_2^o A_3^o A_4^o)$ contains the segment $(B_{13}^o B_{24}^o)$ but neither the segment $(B_{12}^o B_{34}^o)$ nor $(B_{14}^o B_{23}^o)$. Here $B_{ij}^o$ is the midpoint of $(A_i^o, A_j^o)$.
\end{lem}

\begin{proof}
    The displacement vectors between the opposite points on the octahedron formed by the vertices $B_{ij}$ are given as follows:
    \begin{align*}
        B_{12} - B_{34} = \bigg( -\frac12, -1, -\frac12 \bigg), \>\> B_{13} - B_{24} = \bigg(-\frac12, 0, -\frac12 \bigg), \>\> B_{14} - B_{23} = \bigg( \frac12, 0, -\frac12 \bigg)
    \end{align*}
    Denote $B_{\bullet} = \{B_{ij}\}$ as the set of 6 vertices of the octahedron. Denote $\pi_{ij|kl}$ as the unique plane passing through the 4 points $B_{\bullet} \backslash \{ B_{ij}, B_{kl} \}$. Define the vectors $N_{ij|kl}$ as follows (using cross product):
    \begin{align*}
        N_{12|34} =& (B_{13} - B_{24}) \times (B_{14} - B_{23}) = \bigg( 0, -\frac12, 0 \bigg) \\
        N_{13|24} =& (B_{12} - B_{34}) \times (B_{14} - B_{23}) = \bigg( \frac12, -\frac12, \frac12 \bigg) \\
        N_{14|23} =& (B_{12} - B_{34}) \times (B_{13} - B_{24}) = \bigg( \frac12, 0, - \frac12 \bigg)
    \end{align*}
    Since the common midpoints the three pairs of antipodal vertices of the octahedron is $B_0 \eqdef (1/4, 1/2, 3/4)$, the equations for $\pi_{ij|kl}$ are given as follows:
    \begin{align*}
        \pi_{12|34} =& \bigg\{ X \>\bigg|\> N_{12|34} \cdot (X-B_0) = 0 \bigg\} = \bigg\{ (x,y,z) \>\bigg|\> y = \frac12 \bigg\} \\
        \pi_{13|24} =& \bigg\{ X \>\bigg|\> N_{13|24} \cdot (X-B_0) = 0 \bigg\} = \bigg\{ (x,y,z) \>\bigg|\> x -y +z = \frac12 \bigg\} \\
        \pi_{14|23} =& \bigg\{ X \>\bigg|\> N_{14|23} \cdot (X-B_0) = 0 \bigg\} = \bigg\{ (x,y,z) \>\bigg|\> x - z = -\frac12 \bigg\}
    \end{align*}
    Then it is easy to see that the planes $\pi_{12|34}$ and $\pi_{14|23}$ belong to the 2nd Coxeter subdivision, while $\pi_{13|24}$ does not. This concludes the proof.
\end{proof}

From the above, we get the following issue.
\begin{lem} \label{lem:cox_bad2}
    Let $A_1, A_2, \ldots A_6 \in \RR^5$ be points in a general position. Let $K_1, K_2$ be the 2nd (ordered) FK triangulation of the ordered simplices $\sigma_1, \sigma_2$, where $\sigma_i$ are given by $\sigma_1 =(A_1 A_2 A_3 A_4 A_5)$ and $\sigma_2 =(A_2 A_3 A_4 A_1 A_6)$. Then the simplicial complexes $K_1, K_2$ restrict to two different complexes in the 3-simplex $(A_1 A_2 A_3 A_4)$. In other words, $K_1, K_2$ cannot be glued at their common (3-dimensional) face to form a single simplicial complex.
\end{lem}
\begin{proof}
    We first recall that FK triangulation is compatible with its faces, in that a FK triangulation of a $d$-simplex $\sigma = (A_1 A_2 \ldots A_{d+1})$ restricts to the $d'$-face $\tau = (A_{i(1)}, \ldots A_{i(d'+1)})$ as the FK triangulation of $\tau$ (here $i(1) < \cdots < i(d'+1)$). 
    
    Let $\tau_1 = (A_1 A_2 A_3 A_4)$ and $\tau_2 = (A_2 A_3 A_4 A_1)$ be (distinct) ordered 3-simplices. Then due to the aforementioned face-compatibility, $\sigma_i$ restricts to $\tau_i$ for $i=1, 2$. However at this point we see from Lemma \ref{lem:cox_bad1} that $\tau_1$ is \emph{not} cut along the bisecting plane of $(B_{13} B_{24})$, whereas $\tau_2$ is cut along the bisecting plane of $(B_{13} B_{24})$. This concludes the proof.
\end{proof}

\section{Necessity of simplex quality}
\label{appendix: simplex_quality}

Here we use a concrete example to illustrate that the simplex quality needs to be controlled in Proposition \ref{prop:metric_distortion_rectification}.

\begin{prop}
    Let $\epsilon \in (0, 1/2)$ be a number. Let $\sigma_\epsilon \subset \RR^d$ be a simplex whose vertices are $(\vecbf u_0, \ldots \vecbf u_d)$, given as follows:
    \begin{align*}
        \vecbf u_0 = \vecbf 0_d \text{, and } \vecbf u_i = (\vecbf u_i', \epsilon) \text{, for } 1 \le i \le d
    \end{align*}
    where $(\vecbf u_1', \ldots \vecbf u_{d-1}')$ is a basis of $\RR^{d-1}$ and $\vecbf u_d' = -(\vecbf u_1' + \cdots + \vecbf u_{d-1}')$. Let $\phi: \ball_{\RR^d}(\vecbf 0, 1) \rightarrow \SS^d$ be an inverse-projection of a sphere, defined as follows:
    \begin{align*}
        \phi(\vecbf x) & = (\vecbf x, \sqrt{1 - \|\vecbf x\|^2})
    \end{align*}
    Suppose that the norms are bounded as $\|\vecbf u_i'\| < 1/2$ for all $1 \le i \le d$, so that $\phi$ is well-defined on the simplex $\sigma_\epsilon$. Then the following limit holds:
    \begin{align*}
        \lim_{\epsilon \rightarrow 0^+} \frac{\|\diff_{\vecbf 0}\phi_\triangle(\vecbf v) \|}{\|\diff_{\vecbf 0}\phi (\vecbf v) \|} = +\infty
    \end{align*}
    where $\vecbf v = \vecbf e_d$ is the $d$-th standard basis vector.
\end{prop}
\begin{proof}
Firstly $(\vecbf u_1, \ldots \vecbf u_d)$ is a basis of $\RR^d$. Explicitly, any $\vecbf x = (\vecbf x', x_d) \in \RR^d$ can be written as:
\begin{align*}
    \vecbf x =& (\vecbf U' (\vecbf U')^{-1} \vecbf x', x_d) \\
    =& (a_1 \vecbf u_1' + \cdots + a_{d-1} \vecbf u_{d-1}', x_d) \text{, where } \vecbf a = (a_1, \ldots a_{d-1}) = (\vecbf U')^{-1} \vecbf x' \\
    =& a_1 \vecbf u_1 + \cdots + a_{d-1} \vecbf u_{d-1} + (x_d - \epsilon \langle \vecbf a, \vecbf 1_{d-1} \rangle ) \vecbf e_d \\
    =& (a_1 + b) \vecbf u_1 + \cdots + (a_{d-1} + b) \vecbf u_{d-1} + b \vecbf u_d \text{, where } b = \frac{x_d - \epsilon \langle \vecbf a, \vecbf 1_{d-1} \rangle}{d \cdot \epsilon}
\end{align*}
We note that $\sum_{i=1}^d \vecbf u_i = (\sum_{i=1}^d \vecbf u_i' , d \cdot \epsilon) = d \cdot \epsilon \cdot \vecbf e_d$.

The differential of $\phi$ is given by:
\begin{align*}
    \Diff_x \phi (\vecbf v) = \bigg( \vecbf v, \frac{- \langle \vecbf v, \vecbf x \rangle}{\sqrt{1 - \|\vecbf x \|^2}} \bigg)
\end{align*}
In particular, the differential at the origin is simply $\diff_{\vecbf 0} \phi(\vecbf v) = (\vecbf v, 0)$.

Define the following vectors:
\begin{align*}
    \vecbf p =& (\sqrt{1 - \|\vecbf u_1\|^2}, \ldots \sqrt{1 - \| \vecbf u_d \|^2}) \\
    \vecbf p' =& (\sqrt{1 - \|\vecbf u_1\|^2}, \ldots \sqrt{1 - \| \vecbf u_{d-1} \|^2})
\end{align*}
The secant map and its derivative (which is itself) are given as follows. Given a (tangent) vector $\vecbf v  = (\vecbf v', v_d) \in \RR^d$ where $\vecbf v' \in \RR^{d-1}$, we have:
\begin{align*}
    \phi_\triangle'(\vecbf v) =& \phi_\triangle(\vecbf v', v_d) \\
    =& \phi_\triangle \big( (a_1 + b) \vecbf u_1 + \cdots + (a_{d-1} + b) \vecbf u_{d-1} + b \vecbf u_d \big) \\
    =& (a_1 + b) \phi (\vecbf u_1) + \cdots + (a_{d-1} + b) \phi (\vecbf u_{d-1}) + b \phi(\vecbf u_d) \\
    =& (\vecbf v, c )
\end{align*}
where
\begin{align*}
    \vecbf a =& (a_1, \ldots a_{d-1}) = (\vecbf U')^{-1} \vecbf v', \quad
    b = \frac{v_d - \epsilon \langle \vecbf a, \vecbf 1 _{d-1} \rangle }{d \cdot \epsilon}, \quad
    c = \langle \vecbf a, \vecbf p' \rangle + b \cdot \|\vecbf p \|_1
\end{align*}
This implies that:
\begin{align*}
    \frac{\| \diff_{\vecbf 0} \phi_\triangle(\vecbf v) \|^2 }{\| \diff_{\vecbf 0} \phi(\vecbf v) \|^2 } = 1 + \frac{c^2}{\|\vecbf v \|^2}
\end{align*}

In particular, when we set $\vecbf v = \vecbf e_d$, we have that:
\begin{align*}
    \vecbf a =& (\matbf U')^{-1} \vecbf 0_{d-1} = \vecbf 0_{d-1} \\
    b =& \frac{1 - \epsilon \cdot 0}{d \cdot \epsilon} = \frac1{d \cdot \epsilon} \\
    c =& 0 + b \cdot \|\vecbf p\|_1 = \frac{\| \vecbf p\|_1}{d \cdot \epsilon}
\end{align*}
Here, since $\sqrt{s^2-t^2} \ge s-t$ whenever $t \in [0, s]$, the following lower bound holds assuming that $\epsilon \le \min_i \sqrt{1 - \|\vecbf u_i'\|^2}$:
\begin{align*}
    \|\vecbf p\|_1 = \sum_{i=1}^d \sqrt{1 - \|\vecbf u_i\|^2} = \sum_{i=1}^d \sqrt{1 - \|\vecbf u_i' \|^2 - \epsilon^2} \ge \bigg(\sum_{i=1}^d \sqrt{1 - \|\vecbf u_i' \|^2}\bigg) - d \cdot \epsilon
\end{align*}

Therefore,
\begin{align*}
    \frac{\| \diff_{\vecbf 0} \phi_\triangle(\vecbf v) \|^2 }{\| \diff_{\vecbf 0} \phi(\vecbf v) \|^2 } = 1 + \frac{c^2}{\|\vecbf v \|^2} = 1 + \frac{\sum_{i=1}^d \sqrt{1 - \|\vecbf u_i' \|^2 - \epsilon^2}}{d \cdot \epsilon} \ge \frac{\sum_{i=1}^d \sqrt{1 - \|\mathbf u_i'\|^2}}{d \cdot \epsilon}
\end{align*}
which diverges as $\epsilon \rightarrow 0^+$. Note that $\|\vecbf u_i'\|$ are fixed regardless of $\epsilon$.
\end{proof}

\end{document}